\documentclass{article}

\usepackage{amsmath}
\usepackage{amssymb}
\usepackage{amsthm}
\usepackage{mathtools}
\usepackage{geometry}

\newtheorem{thm}{Theorem}[section]
\newtheorem{prop}[thm]{Proposition}

\theoremstyle{definition}
\newtheorem*{ex}{Example}
\newtheorem*{rem}{Remark}
\newtheorem*{rems}{Remarks}
\newtheorem*{notation}{Notation}

\DeclareMathOperator{\Div}{div}
\DeclareMathOperator{\rot}{rot}

\newcommand{\NN}{\mathbb{N}}
\newcommand{\sU}{\mathcal{U}}
\newcommand{\sV}{\mathcal{V}}
\newcommand{\sW}{\mathcal{W}}

\newcommand{\sP}{\mathcal{P}}
\newcommand{\sQ}{\mathcal{Q}}
\newcommand{\sS}{\mathcal{S}}
\newcommand{\sF}{\mathcal{F}}
\newcommand{\RR}{\mathbb{R}}
\newcommand{\innpr}{\cdot}

\begin{document}

\title{Steady three-dimensional rotational flows: an approach via two stream functions and Nash-Moser iteration}

\author{
B. Buffoni\thanks{Institut de math\'ematiques, Station 8, Ecole Polytechnique
F\'ed\'erale de Lausanne, 1015 Lausanne, Switzerland
} 
\and E. Wahl\'en\footnote{Centre for Mathematical Sciences, Lund University, PO Box 118, 22100 Lund, Sweden}
}

\date{December 20, 2018}

\maketitle

\begin{abstract}
We consider the stationary flow of an inviscid and incompressible fluid of constant density in the region $D=(0, L)\times \RR^2$. We are concerned with flows 
that are periodic in the second and third variables
and that have prescribed flux through each point of the boundary $\partial D$.
The Bernoulli equation states that the 
``Bernoulli function''
$H\coloneqq \frac 1 2 |v|^2+p$ (where $v$ is the velocity field and $p$ the pressure) is constant along stream lines, that is, 
each particle is associated with a particular value of $H$. We also prescribe the value of $H$ on $\partial D$.
The aim of this work is to develop an existence theory
near a given constant solution. It relies on writing the velocity field in the form $v=\nabla f\times \nabla g$ and deriving a degenerate nonlinear elliptic system for $f$ and $g$. This system is solved using the Nash-Moser method, as developed for the problem
of isometric embeddings of Riemannian manifolds;
see e.g.\ the book by Q. Han and J.-X. Hong (2006). 
Since we can allow $H$ to be non-constant on $\partial D$, our theory includes three-dimensional flows with non-vanishing vorticity.
\smallskip


\noindent
{\em Mathematics subject classification (AMS, 2010):}
35Q31, 76B03, 76B47, 35G60, 58C15.

\end{abstract}

\section{Introduction}
The Euler equation for an inviscid and incompressible fluid
of constant density is given by
$$(v\innpr \nabla)v=-\nabla p,~~\Div v=0,$$
if in addition the velocity field $v$ is independent of time.
As we are concerned with stationary flows on $D=(0,L)\times \RR^2$
that are periodic in the second and third variables, it is useful
to introduce the cell of the periodic lattice
$$\sP=(0,L)\times (0,P_1)\times(0,P_2),$$
where $L>0$ and the periods $P_1, P_2>0$ are given; in particular integrations
will mainly be over $\sP$ and maxima of continuous functions considered on $\overline{\sP}$.
Any constant vector field $\bar v$ is a solution on 
$D$ with constant pressure $\bar p$.
Such a field can always be written in the form $\bar v=\nabla \bar f\times \nabla \bar g$,
for some linear functions $\bar f, \bar g$.
If the real-valued functions 
$$
(x,y,z)\mapsto f_0(x,y,z),\quad
 (x,y,z)\mapsto g_0(x,y,z),
\quad (x,y,z)\in D,
$$
are near $0$ and $(P_1,P_2)$-periodic in $(y,z)$, 
one may try looking for a velocity field
of the form
$$v^*=\nabla (\bar f +f_0+f^*)\times\nabla (\bar g +g_0+g^*)$$ 
for unknown functions $f^*$ and $g^*$
that vanish at the boundaries $x=0$ and $x=L$. The functions $f_0$ and $g_0$
can be interpreted
 as encoding a perturbation of the boundary conditions at $x=0$ and $x=L$
given by $\bar f$ and $\bar g$.
If $f_0$ and $g_0$ vanish at $x=0$ and $x=L$, then nothing is gained
with respect to the case $f_0=g_0=0$ on $D$.

In the following theorem, 
the Sobolev spaces $W^{n,p}_{loc}(D)$ and $H^n_{loc}(D)$ 
consist of functions defined on $D$ such that, 
when restricted to every bounded open subset $D_b\subset D$, 
they belong to $W^{n,p}(D_b)$ and $H^n(D_b)$. Note that, in contrast with the usual 
definition, $\overline {D_b}$ is not required to be included in $D$. Moreover, $\sQ$ is the parallelogram in $\RR^2$ spanned by $RP_1 e_1$ and  $RP_2e_2$, where
\[
R=\begin{pmatrix}
\partial_2 \bar f & \partial_3 \bar f\\
\partial_2 \bar g & \partial_3 \bar g
\end{pmatrix},
\]
is the Jacobian matrix of $(\bar f,\bar g)$ with respect to $(y,z)$ and $\NN_0=\{0,1,2,\ldots\}$.
\begin{thm}
\label{thm: main in intro}
Let $j\in \NN_0$ and assume that the first component of $\bar v$ does not vanish. 
Then it is possible to choose $\bar \epsilon>0$ such that if 
\begin{itemize}
\item[$\bullet$]
$H_0\in C^{11+j}(\RR^2)$ is periodic with respect to the lattice in $\RR^2$ generated by $RP_1 e_1$ and $RP_2e_2$
(not necessarily the fundamental periods, this remark holding generally throughout),
\item[$\bullet$]
$c_1, c_2\in \RR$,
\item[$\bullet$]
$f_0,g_0 \in H^{13+j}_{loc}(D)=W^{13+j,2}_{loc}(D)$,  $P_1$-periodic in $y$ and $P_2$-periodic in $z$,
\item[$\bullet$]
$\displaystyle \|(f_0,g_0)\|_{H^{13+j}(\sP)}^2+\|H_0\|^2_{C^{11+j}(\overline \sQ)}+|c|^2 <\bar \epsilon^2$,
\end{itemize}
then there exists
$(f^*,g^*)\in H^{6+j}_{loc}(D)$ satisfying
\begin{itemize}
\item[$\bullet$]
$ f^*,g^*$ are $P_1$-periodic in $y$
and $P_2$-periodic in $z$,
\item[$\bullet$]
\begin{equation}
\text{$f^*,g^*$ vanish when $x\in\{0,L\}$,}
\label{eq: bdy cdn}
\end{equation}
\item[$\bullet$]
$v^*\coloneqq \nabla (\bar f +f_0+f^*)\times\nabla (\bar g +g_0+g^*)$
is a solution to the  Euler equation 
$$(v^*\innpr \nabla)v^*=-\nabla p^*,~~\Div v^*=0~\text{ on }~D,$$
with
\begin{equation}
 p^*=-\frac 1 2| v^*|^2 +H(\bar f+f_0+f^*, \bar g + g_0 +g^*)
~\text{ and }~H(f,g)=c_1f+c_2 g+H_0(f,g)~\text{for all } f,g\in\RR.
\label{eq: Bernoulli eqn}
\end{equation}
\end{itemize}
Moreover, there exists a constant $C>0$ (independent of $(f_0, g_0)$, $H_0$ and $c$) such that 
$$\|(f^*, g^*)\|_{H^{6+j}(\sP)}\le C\bar \epsilon.$$
The solution is locally unique in the following sense. Let $H$ be as above (but $H_0$ can be assumed of class $C^2$ only),  $f,g, \widetilde f, \widetilde g\in C^3(\overline D)$ with $(f-\bar f, g-\bar g)$, $(\widetilde f-\bar f, \widetilde g-\bar g)$ both  $(P_1, P_2)$-periodic in $y$ and $z$, and 
$$
(f(x,y,z),g(x,y,z))=(\widetilde f(x,y,z),\widetilde g(x,y,z)),~\text{for all } (x,y,z)\in \{0,L\}\times \RR^2.
$$
Assume  that  $v=\nabla f\times \nabla g$ and $\widetilde v=\nabla \widetilde f\times \nabla \widetilde g$ are both solutions to the Euler equation with pressures $-\frac12 |v|^2+H(f,g)$ and $-\frac12 |\widetilde v|^2+H(\widetilde f,\widetilde g)$, respectively. 
If $(\nabla f,\nabla g)$ and
$(\nabla \widetilde f,\nabla \widetilde g)$ are 
in a sufficiently small open convex neighborhood of $(\nabla \bar f,\nabla \bar g)$
in  $C^2(\overline \sP)$ and $\|H_0\|_{C^2(\overline \sQ)}$ is sufficiently small,  
then $(f,g)=(\widetilde f,\widetilde g)$ on $[0,L]\times \RR^2$.
\end{thm}

\begin{rems}$ $

\begin{itemize}
\item
Observe that $\nabla_{(f,g)} H(\bar f+f_0+f^*, \bar g+g_0+g^*)$  is $P_1$-periodic in $y$
and $P_2$-periodic in $z$.  
In general the choice $(f^*,g^*)=-(f_0,g_0)$ is not allowed,
as $(f^*,g^*)$ is required to vanish at $x=0$ and $x=L$, but not
$(f_0,g_0)$.
When $H$ is constant, the choice $(f^*, g^*)=-(f_0, g_0)$ leads to the constant solution $v^*=\bar v$,
provided that $f_0$ and $g_0$ vanish when $x\in\{0,L\}$.
However, when $H$ is not constant 
\eqref{eq: bdy cdn} and \eqref{eq: Bernoulli eqn}
do not allow to choose $(f^*,g^*)=-(f_0,g_0)$. Indeed, if $(f^*,g^*)=-(f_0,g_0)$, then $v^*=\bar v$ and $p^*$ should be constant, which is not compatible with \eqref{eq: Bernoulli eqn} when $H$ is not constant.
\item
If $H_0$, $f_0$ and $g_0$ are $C^\infty$ smooth, we obtain solutions of arbitrarily high regularity. However, we don't necessarily obtain $C^\infty$ smooth solutions since $\overline{\epsilon}$ depends on $j$. 
It might be possible to obtain smooth solutions by applying
other versions of the Nash-Moser theorem, for example an
analytic version, but that's outside the scope of the paper.
\item The uniqueness assertion implies that the solution $(\bar f+f_0+f^*, \bar g+g_0+g^*)$ only depends on $f_0$ and $g_0$ through their boundary values.
\item On the other hand, it is possible for two different sets of data to give rise to the same velocity field $v$ (see the Appendix for more details).
\end{itemize}
\end{rems}

The following example illustrates the relationship
with Beltrami flows (flows such that, at each point of $D$, 
the vorticity is parallel to the velocity) and the role of the boundary conditions at $x=0$ and $x=L$.

\begin{ex}
Let $\bar f(x,y,z)=y$, $\bar g(x,y,z)=z$, $c_1,c_2=0$ and $H_0=0$,
so that $\bar v=(1,0,0)$.
Let $f_0(x,y,z)=\delta x\sin(2\pi z/P_2)$ and $g_0=0$, and
let $(f^*,g^*)$ be given by Theorem \ref{thm: main in intro} 
(for $|\delta|$ small enough).
Remember that $f^*$ and $g^*$ vanish at $x=0$ and $x=L$.
The pointwise flux of $v^*$ at $x=0$ and $x=L$ is the constant $1$:
$$v_1^*=\partial_y(\bar f+f_0)\partial_z(\bar g+g_0)-
\partial_z(\bar f+f_0)\partial_y(\bar g+g_0)=1.$$
Let us prove that $v^*$ is not irrotational by assuming the opposite.
Then $v^*_1$ would be a $(P_1,P_2)$-periodic function in $y$ and $z$
that is harmonic. By the maximum principle, $v^*_1=1$ and
thus $(v^*_2,v^*_3)$ would be $x$-independent.
The functions $v_2^*$ and $v^*_3$ would also be harmonic and thus
they would be constant, and $v^*$ would be a constant vector field. 
Hence the map that sends a fluid parcel when $x=0$ to its position when $x=L$
would be a translation. But this is impossible because $\bar f+f_0+f^*$ is preserved
along every parcel trajectory and its level sets  at $x=0$
(that is, the level sets of $\bar f+f_0$ at $x=0$) 
cannot be sent by a translation to its level sets at $x=L$.
Although $v^*$ is not an irrotational flow, it is a Beltrami flow
because $H=0$. As the flux through the boundaries $x=0$ and $x=L$
does not vanish,
the proportionality factor between the velocity and the vorticity
cannot be constant (using also the periodicity in the $y$ and $z$ directions).
Beltrami flows have been considered in many papers, for example
in \cite{EnPe:2015} (Beltrami flows with constant proportionality factors)
and \cite{KaNeWa:2000} 
(with non-constant proportionality factors).
\end{ex}

The representation $v=\nabla f\times \nabla g$ can be seen as a generalization 
of the stream function representation $v=\nabla^\perp \psi$ for planar divergence-free stationary flows, in which the stream function $\psi$ is replaced by a pair of functions $f$ and $g$ (note that $f$ and $g$ are constant on stream lines). This representation  always holds locally near regular points of the velocity field (see, e.g., \cite{Barbarosie}). 
For the reader's convenience,
we give in the Appendix a self-contained proof when $v_1$ is non-vanishing
 that the representation holds globally in $D$  with additional $(P_1,P_2)$-periodicity with respect to $y$ and $z$ for $\nabla f$ and $\nabla g$.

In this formulation, the Euler equation has a particularly helpful variational structure \cite{Ke} (see also \cite{Bu:2012}).
Namely, the pair of functions $(f,g)$ will be called admissible for the present purpose if
\begin{itemize}
\item
$f$ and $g$ are  of class $C^2(\overline D)$,
\item
$\nabla f$ and $\nabla g$ are $P_1$-periodic in $y$ and $P_2$-periodic in $z$,
\item
$\displaystyle~
(f(x,y,z),g(x,y,z))=(\widetilde f_0(x,y,z),\widetilde g_0(x,y,z)),~\text{for all } (x,y,z)
\in \{0,L\}\times \RR^2,$
\end{itemize}
where $\widetilde f_0$ and $\widetilde g_0$ are two fixed functions
of class $C^2(\overline D)$
such that $\nabla \widetilde f_0$ and $\nabla \widetilde g_0$ 
are $P_1$-periodic in $y$ and $P_2$-periodic in $z$.
Under these conditions,
$v=\nabla f\times \nabla g$ is divergence free and the first component
$$v_1=(\nabla f\times\nabla g)\innpr (1,0,0)
=\partial_y f\, \partial_z g-\partial_yg\, \partial_z f
=\partial_y \widetilde f_0\,\partial_z \widetilde g_0-\partial_y\widetilde g_0\,
\partial_z \widetilde f_0$$
of $v$ is prescribed
on $\{0,L\}\times \RR^2$. 
In order to get a better insight into the set of admissible 
$(f,g)$,
note that $f(x,y,z)-a_1y-a_2z$ and
$g(x,y,z)-a_3y-a_4z$  are 
$P_1$-periodic in $y$ and $P_2$-periodic in $z$ for some constants
$a_1,a_2,a_3,a_4\in \RR$. The boundary condition ensures that
$a_1,a_2,a_3,a_4\in \RR$ do not depend on the particular admissible 
pair of functions
$(f,g)$.

We also assume that
the function $H\colon \RR^2\rightarrow \RR$ is of class $C^2$ and
that $\partial_f H$ and $\partial_g H$  composed with every admissible
pair  $(f,g)$ are $(P_1,P_2)$-periodic in $y$ and $z$. The latter is equivalent to 
requiring that $\nabla_{(f,g)} H$ is periodic with respect to the lattice 
generated by $P_1(a_1, a_3)$ and $P_2(a_2, a_4)$.

Let $(\widetilde f,\widetilde g)$ be admissible
and assume that $(\widetilde f,\widetilde g)$ is a critical point of
the integral functional
\begin{equation}
\label{eq: integral functional}
\int_{\sP}\Big\{ \frac12|\nabla f\times \nabla g|^2+H(f,g)\Big\}\,dx\,dy\,dz.
\end{equation}
defined on the set of admissible pairs $(f,g)$.
Let us check that $\widetilde v \coloneqq \nabla \widetilde f\times \nabla \widetilde g$
is a solution to the Euler equation with 
$\widetilde p=-\frac 1 2|\widetilde v|^2
+H(\widetilde f,\widetilde g)$.
We consider admissible variations $(f_s,g_s)$,
that is, maps $(s,x,y,z)\rightarrow (f_s(x,y,z),g_s(x,y,z))$
of class $C^2([-1,1]\times  \overline D)$
such that $(f_0,g_0)=(\widetilde f,\widetilde g)$,
$(f_1,g_1)$ is admissible and 
$$(f_s,g_s)=\Big((1-s)f_0+sf_1\,,\,(1-s)g_0+sg_1\Big)
~~\text{ for all }~~s\in(-1,1).$$
The meaning of critical point is that the integral functional
at $(f_s,g_s)$ as a function of $s$ has a vanishing derivative at $s=0$,
for every admissible variation $(f_s,g_s)$.
If in addition we assume that $(f_1-f_0,g_1-g_0)$ is compactly supported in $\sP$,
we get the Euler-Lagrange equation
\begin{equation}\label{eq: deux}
\left(\begin{array}{c}
-\Div(\nabla \widetilde g \times (\nabla \widetilde f\times \nabla \widetilde g))+
\partial_f H(\widetilde f,\widetilde g)
\\
-\Div((\nabla \widetilde f\times \nabla \widetilde g)\times \nabla \tilde f))
+\partial_gH(\tilde f,\tilde g)
\end{array}\right)=0.
\end{equation}
Because of the periodicity assumption on 
$\nabla\widetilde f$ and $\nabla\widetilde g$,
more general admissible variations $(f_s,g_s)$  do
not provide additional knowledge and,
thanks to the periodicity condition on
$\partial_fH(\widetilde f,\widetilde g)$  and
$\partial_gH(\widetilde f,\widetilde g)$, \eqref{eq: deux} holds true on all of 
$D$.
Equation \eqref{eq: deux} can also be written
\begin{equation}
\label{eq: deux bis} 
\nabla \tilde g \innpr \rot \tilde  v+\partial_fH(\tilde f,\tilde g)=0
\hbox{ and }
- \rot \tilde v \innpr \nabla \tilde f+\partial_g H(\tilde f,\tilde g)=0,
\hbox{ with }\tilde v=
\nabla \tilde f \times \nabla \tilde g.
\end{equation}
It then follows that
\begin{equation}
\label{eq: invariance of H}
\begin{aligned}
\tilde v \times \rot\tilde v
&=(\nabla \tilde f\times\nabla \tilde g)\times \rot\tilde v=(\nabla \tilde f
\innpr \rot\tilde v)\nabla \tilde g-
(\nabla \tilde g\innpr \rot\tilde v)\nabla \tilde f
\\
&=\partial_f H(\tilde f,\tilde g)\nabla \tilde f
+\partial_g H(\tilde f,\tilde g)\nabla \tilde g
=\nabla_{(x,y,z)}H(\tilde f,\tilde g).
\end{aligned}
\end{equation}
The identity (see e.g.\ p.\ 151 in \cite{Se:59})
$$\nabla (\frac 1 2 |\tilde v|^2)
=\tilde v \times \rot\tilde v+( \tilde v \innpr\nabla)\tilde  v$$
gives
$$( \tilde v\innpr \nabla)\tilde v 
-\nabla (\frac 1 2 |\tilde v|^2)+\nabla_{(x,y,z)} H(\tilde f,\tilde g)=0,$$
which is equivalent to the classical Euler equation for
inviscid, incompressible and time-independent flows
$$(\tilde v\innpr \nabla)\tilde v +\nabla \tilde p=0 \hbox{ with }
\tilde p=-\frac 1 2|\tilde v|^2 +H(\tilde f,\tilde g).$$
$H(\tilde f,\tilde g)$ can be seen as the Bernoulli function,
which is preserved by the flow since $\nabla_{(x,y,z)}(H(\tilde f,\tilde g))\innpr \tilde v=0$ by \eqref{eq: invariance of H}.

The aim of the paper is to develop an existence theory
in a small neighborhood of $(\bar f,\bar g)\in C^\infty(\overline D)$ when 
\begin{itemize}
\item $\nabla \bar f$ and 
$\nabla \bar g$ are constant, and
\item the first component of $\bar v
=\nabla \bar f\times\nabla \bar g$ does not vanish. 
\end{itemize}

If we perturb \eqref{eq: deux} into the equation
\begin{equation*}
\left(\begin{array}{c}
-\epsilon (\partial^2_y\widetilde f+\partial^2_z\widetilde f)
-\Div(\nabla \widetilde g \times (\nabla \widetilde f\times \nabla \widetilde g))+
\partial_f H(\widetilde f,\widetilde g)
\\
-\epsilon (\partial^2_y\widetilde g+\partial^2_z\widetilde g)
-\Div((\nabla \widetilde f\times \nabla \widetilde g)\times \nabla \tilde f))
+\partial_gH(\tilde f,\tilde g)
\end{array}\right)=0
\end{equation*}
and then linearize  this perturbed equation,
the obtained linear problem
is coercive \cite{KoNi}, provided that $\epsilon>0$.
The linearization of \eqref{eq: deux} can thus be described as ``degenerate'',
the $x$ direction being however non-degenerate
\cite{KoNi}. 
In Section \ref{sec: linearization}, we analyze the linear operator 
obtained from the linearization of \eqref{eq: deux} and its
invertibility, following the classical work by Kohn and Nirenberg
\cite{KoNi} for non-coercive  boundary value problems. 
The analysis of the linearized problem relies on the particular structure of the integral
functional \eqref{eq: integral functional}. The main point is that its
quadratic part is positive definite (see Proposition \ref{prop: first estimate}  for a precise statement).
The local uniqueness result is obtained as a corollary.

The Nash-Moser iteration method \cite{Mo,Ze}
 has been applied
to non-coercive problems in previous works, like \cite{Ki,Han-Hong}.
The approach we shall follow is 
the one described in Section 6 of \cite{Han-Hong}
for the embedding problem of Riemannian manifolds with non-negative 
Gauss curvature. The details are given in Section \ref{sec: Nash-Moser}. For simplicity, we have restricted ourselves
as in \cite{Han-Hong} to periodicity conditions with respect to $(y,z)$. 
A key ingredient are tame estimates for the inverse of the linearization, which are obtained in Section \ref{sec: linearization} using suitable commutator estimates.

In \cite{Alber}, Alber deals with a closely related setting.
The steady Euler equation is considered in a bounded, simply connected, 
smooth domain $\Omega\subset \RR^3$. There are three boundary conditions:
1) the  flux  through $\partial \Omega$ is given by  a function 
$f\colon\partial\Omega\rightarrow \RR$, 2) a condition on the vorticity
flux through the entrance  set $\{(x,y,z)\in \RR^3:f(x,y,z)<0\}\coloneqq \partial\Omega_-\,$ and 3) a condition on the Bernoulli function
on $\partial\Omega_-$.
Under precise assumptions, existence and uniqueness are obtained 
near a solution $v_0$ with small vorticity when the boundary conditions 2)
and 3) are slightly modified.
In the present paper, boundary condition 2) is,  roughly speaking,
replaced by a condition
on the Bernoulli function on the exit set. 
These more symmetric boundary conditions might be a first step to considering flows which are periodic in $x$, which is a natural geometry in the study of water waves. Our approach also has the benefit of using a variational structure.

Note that the stationary Euler equation also appears as a model in ideal magnetohydrodynamics, with $v$ replaced by  
the magnetic field $B$, the vorticity $\rot v$ replaced by the current density $J$ (up to a constant multiple) and the Bernoulli function $H$ replaced by the negative of the fluid pressure $p$. Grad \& Rubin \cite{GrRu} derived a variational principle for this problem which is rather close to the one considered here (see e.g.\ Theorem 1 in \cite{GrRu}), although they did not use it to construct solutions. 
Moreover the above example is related to their Theorems 3 and 5 and to a remark that follows their Theorem 5.
A recent work that relies on this variational principle for Euler flows is
\cite{Sl:2015};  it is formulated in a more general geometric framework.
An iterative method, not of Nash-Moser type, is developed in
\cite{KaNeWa:2000} to get Beltrami flows with non-constant proportionality
factors. The boundary conditions there have the same flavor as the ones 
in \cite{Alber}.
Writing a divergence-free  velocity field $v$ in the form
$v=\nabla f\times \nabla g$ may also be useful for irrotational flows,
as it could lead to helpful changes of variables; see \cite{Pl:1980}.

\section{Linearization}
\label{sec: linearization}

The variational structure of \eqref{eq: deux} allows one to
study its linearization with the help of
the quadratic part of the integral functional 
\eqref{eq: integral functional}
around an admissible pair $(f,g)$.  
From now on we shall call a pair $(f,g)$ admissible if
\begin{description}
\item[(Ad1)]
$f$ and $g$ are  of class $C^3(\overline D)$,
\item[(Ad2)]
$\nabla f$ and $\nabla g$ are $(P_1,P_2)$-periodic in $y$ and $z$.
\end{description}
The quadratic part is given by
\begin{multline*}
(F,G)\mapsto  \int_{\sP}\Big\{ \frac12
|\nabla F\times \nabla g+\nabla f\times \nabla G|^2+\left( \nabla f\times \nabla g\right)\innpr\left(\nabla F\times \nabla G\right)\\
+\frac12 \big(\partial_f^2 H(f,g)F^2+2\partial_f \partial_g H(f,g) FG+ \partial_g^2 H(f,g) G^2\big)\Big\}\,dx\,dy\,dz,
\end{multline*}
where $(F,G)$ is assumed admissible in the sense that
\begin{description}
\item[(Ad'1)]
$F$ and $G$ are  in the Sobolev space $H^1_{loc}(D)$.
\item[(Ad'2)]
$F$ and $G$ are $(P_1,P_2)$-periodic in $y$ and $z$,
\item[(Ad'3)]
$(F,G)=0$ on $\partial D$ in the sense of traces.
\end{description}
Condition (Ad'3) is introduced because
we shall assume later that the restriction of $(f,g)$ to $\partial D$
is a priori given.

Given an admissible pair $(f,g)$, we shall call $H$ admissible if
\begin{description}
\item[(Ad'')] 
$H\in C^2(\RR^2)$ and $H''(f,g)$ is 
$(P_1,P_2)$-periodic in $y$ and $z$.
\end{description}
In this section we will mostly think of $H''(f,g)$ as a given function of $(x,y,z)$ rather than a composition.

The quadratic part can be written $\frac12 B_{(f,g)}((F,G),(F,G))$,
where $B_{(f,g)}$ is the symmetric bilinear form
\begin{align*}
&
B_{(f,g)}((F,G),(\delta F,\delta G))
\\&=\int_{\sP}\Big\{ 
\left( \nabla F\times \nabla g+\nabla f\times \nabla G\right)\innpr\left(
\nabla \delta F\times \nabla g+\nabla f\times \nabla \delta G\right)
\\&\qquad\qquad+\left( \nabla f\times \nabla g \right)\innpr \left(\nabla F\times \nabla \delta G\right)
+\left( \nabla f\times \nabla g\right)\innpr\left(\nabla \delta F\times \nabla G\right)
\\&\qquad\qquad+
\partial_f^2 H(f,g)F\delta F+\partial_f \partial_g H(f,g) (F\delta G+G\delta F) + \partial_g^2 H(f,g) G \delta G
\Big\}\,dx\,dy\,dz.
\end{align*}

This section contains two kinds of results: 
firstly, we bound from below the quadratic part 
and, secondly, we study the regularity of solutions to the linearization of
problem \eqref{eq: deux}  at $(f,g)$.
A preliminary observation is that the quadratic part is not coercive at $(f,g)$ in the sense that there is no $\alpha>0$ such that, for all admissible $(F,G)$,
$$\frac 1 2 B_{( f, g)}((F,G),(F,G))\geq 
\int_{\sP}\Big\{\alpha(|\nabla F|^2 + |\nabla G|^2)
-\alpha^{-1}(F^2+G^2) \Big\}\,dx\,dy\,dz.
$$
For example, taking $G=0$, the quadratic part becomes
$$
F\mapsto  \int_{\sP}\Big( \frac12
|\nabla F\times \nabla g|^2
+\frac12 \partial_f^2 H(f,g)F^2\Big)\,dx\,dy\,dz.
$$
In the particular case
$f(x,y,z)=y$, $g(x,y,z)=z$, $H=0$ and $P_1=P_2=1$, the integral reduces to
\[
\frac12\int_{\sP} \Big(F_x^2 +F_y^2\Big) \, dx\, dy \, dz.
\]
Choosing $F_n$ of the form
$$F_n(x,y,z)=  \phi(x) \cos(2\pi n z), $$
where $\phi\in C^\infty(\mathbb{R},[0,1])$  is compactly supported in $(0,1)$ and takes the value $1$ on $(1/4,3/4)$, we find that the quadratic part and
$\|(F_n,G)\|_{L^2(\sP)}$ 
have positive constant values along the sequence $\{ (F_n,G) \}_{n\geq 1}$. 
However, $\|(\nabla F_n, \nabla G)\|_{L^2(\sP)}\to \infty$ and thus $\alpha$ as above cannot exist.
For a general pair $(f,g)$, we instead fix $(x_0,y_0,z_0)\in \sP$ such that $\nabla g(x_0, y_0, z_0)\ne 0$
and consider $F_n$ which is $(P_1,P_2)$-periodic in $(y,z)$ and when 
restricted to $\sP$ is given by 
$$F_n(x,y,z)= \phi(x,y,z) \cos\Big(n g(x,y,z)\Big), $$
where $\phi \in C^\infty(\overline{\sP}, [0,1])$ is compactly supported in $\sP$, 
with $\phi(x_0,y_0,z_0)=1$.
By choosing $n$ large enough, 
one again obtains that $\alpha$ cannot exist. 
In fact, we have made the stronger observation  that,
for all $\alpha>0$, there exists a  sequence $\{(F_n,G_n)\}$ of admissible pairs
such that
$$\frac 1 2 B_{( f, g)}((F_n,G_n),(F_n,G_n))
+\alpha^{-1}\int_{\sP}(F_n^2+G_n^2) \,dx\,dy\,dz
$$
remains bounded, but $\{(F_n,G_n)\}$ does not have any subsequence converging in $L^2(\sP)$.
This has implications for the regularity of the solutions to the linearized problem, as described below.

Nevertheless,
in Theorem \ref{thm: quadratic part}, we bound from below the quadratic part
in a rougher way.
The term
$\int_{\sP} \frac12|\nabla F\times \nabla g+\nabla f\times \nabla G|^2\,dx\,dy\,dz$ 
turns out to be rather nice, as shown in the first part
of the proof, because it is bounded from below by 
$\int_{\sP}\left\{(v\innpr \nabla F)^2+(v\innpr \nabla G)^2\right\}\,dx\,dy\,dz$ (under the simplifying assumption  \eqref{eq: for simplicity},
otherwise there is an additional factor). 
With the help of  a Poincar\'e inequality and thanks to the Dirichlet boundary condition at $x=0$ and $x=L$, 
$\int_{\sP}\left\{(v\innpr \nabla F)^2+(v\innpr \nabla G)^2\right\}\,dx\,dy\,dz$ can in turn be bounded from below by a positive constant times $\|(F,G)\|^2_{L^2(\sP)}$.
In the second and third parts of the proof of
Theorem \ref{thm: quadratic part},
we bound from below the second term of the quadratic part, that is,
$\int_{\sP} \left( \nabla f\times \nabla g\right)\innpr\left(\nabla F\times \nabla G\right)\,dx\,dy\,dz$: 
it cannot become too negative with respect to
$\int_{\sP} \frac12|\nabla F\times \nabla g+\nabla f\times \nabla G|^2
\,dx\,dy\,dz$. 
In these estimates, it is assumed that 
$(\nabla f,\nabla g)$ is in some small neighborhood of $(\nabla \bar f,\nabla \bar g)$ in  $C^2(\overline \sP)$.
To get a better feeling for the term
$\int_{\sP} \left( \nabla f\times \nabla g\right)\innpr\left(\nabla F\times \nabla G\right)\,dx\,dy\,dz$, observe that it vanishes when $v$ is irrotational  
because  (see the beginning of the second step)
$$ \int_{\sP}\left(\nabla f\times \nabla g\right)\innpr\left(\nabla F\times \nabla G\right) \,dx\,dy\,dz
=\frac12\int_{\sP}\rot v\innpr\left(F\nabla G-G\nabla F\right)\,dx\,dy\,dz.$$
As we allow $v$ to be slightly rotational, this term needs careful
estimates.

As a consequence of 
Theorem \ref{thm: quadratic part}, the integral functional is
strictly convex in a neighborhood of $(\bar f,\bar g)$, which
implies local uniqueness
of a solution to \eqref{eq: deux} (but not existence at this stage); 
see Theorem \ref{thm: uniqueness}.

With the aim to apply the technique of elliptic regularization \cite{KoNi}, 
we consider for $\epsilon\in[0,1]$ the regularized quadratic part
\begin{align*}
(F,G)\mapsto &\int_{\sP}\Big\{ \frac12
|\nabla F\times \nabla g+\nabla f\times \nabla G|^2
+\left( \nabla f\times \nabla g\right)\innpr\left(\nabla F\times \nabla G\right) 
\\
&
+\frac{\epsilon}{2}(|\nabla F|^2+|\nabla G|^2)
+\frac12 \big(\partial_f^2 H(f,g)F^2+2\partial_f \partial_g H(f,g) FG+ \partial_g^2 H(f,g) G^2\big)\Big\}\,dx\,dy\,dz
\\&:=\frac 1 2 B_{(f,g)}^\epsilon((F,G),(F,G)).
\end{align*}
All the obtained estimates are
uniform in $\epsilon \in [0,1]$, but, in addition, the problem becomes
elliptic for $\epsilon \in (0,1]$.

For every admissible $(f,g)\in C^3(\overline D)$, we introduce
the following system for
$(\mu,\nu)\in L^2_{loc}(D)$ that is $(P_1,P_2)$-periodic in $y$ and $z$, and for
$(F,G)\in H^2_{loc}(D)$
admissible in the sense of (Ad'1)--(Ad'3):
$$
\begin{array}{l}
\displaystyle
\mu=-\Div\Big(\nabla g\times(\nabla F\times\nabla g
+\nabla f\times \nabla G)
+\nabla G\times(\nabla f\times \nabla g)\Big)
\\~~~~~
-\epsilon \Delta F
+\partial_f^2 H(f,g)F+\partial_f \partial_g H(f,g) G,
\\
\displaystyle
\nu=-\Div\Big((\nabla F\times\nabla g
+\nabla f\times \nabla G)\times \nabla f
+(\nabla f\times \nabla g)\times \nabla F\Big)
\\~~~~~
-\epsilon \Delta G
+\partial_f \partial_g H(f,g) F
+\partial_g^2 H(f,g)G.
\end{array}
$$
The right-hand side is the linear operator related
to the regularized quadratic part.
This system also makes sense in a weak form if, 
instead of $(F,G)\in H^2_{loc}(D)$, we ask that $(F,G)\in H^1_{loc}(D)$.
Given $(\mu,\nu)$ in any higher-order Sobolev space, the main issue of
Section \ref{sec: linearization} is to study the regularity of
a solution $(F,G)$, aiming at estimates of the Sobolev norms,
uniformly in $\epsilon\in[0,1]$.
Such a  pair $(F,G)$ is easily proved to be unique and its existence for
$\epsilon\in (0,1]$ follows from the fact that the system is elliptic.
The same particular case as above gives more insight into this system. 
Setting $\mu=\nu=0$, $\epsilon=0$, $G=0$, $f(x,y,z)=y$, $g(x,y,z)=z$ and $P_1=P_2=1$,
we get
$$
\begin{array}{l}
\displaystyle
-\Div(\partial_1F,\partial_2F,0)
+\partial_f^2 H(f,g)F=0,
\\
\displaystyle
-\Div(0,-\partial_3 F, 2\partial_2 F)
+\partial_f \partial_g H(f,g) F=0.
\end{array}
$$
Keeping only the second order terms and forgetting the boundary and
periodicity conditions,  we see that $F(x,y,z)=\cos(z)$ is a solution
to both equations. Hence the regularity theory in \cite{AgmonDouglisNirenberg} cannot be used when
$\epsilon=0$, $f(x,y,z)=y$, $g(x,y,z)=z$ and $P_1=P_2=1$.

In Proposition
\ref{prop: solve second partial in x},
we explain how the general system   allows one
to express 
$\partial^2_{11}F$ and $\partial^2_{11}G$  with respect to
the other second-order partial derivatives of $F$ and $G$, and
lower-order terms, involving $\mu$ and $\nu$ too.
After iterative differentiations, this also yields expressions for higher-order
derivatives that contain at least two partial derivatives
with respect to $x$. In a more general setting, this is developed
in \cite{KoNi}.

For $i\in\{2,3\}$,
multiplying both sides of each equation of the system by 
$(-1)^r\partial_i^{2r} F$ and
$(-1)^r\partial_i^{2r} G$, respectively, 
summing the two equations and then integrating by parts
many times, $B_{(f,g)}(\partial_i^r F,\partial_i^r F)$ 
arises, with additional bilinear terms in $(F,G)$
that turn out to involve  at most $r$ partial
derivatives of $F$ and $G$ for each of the two components
of each bilinear term.
We can make some of  these additional terms small if $v$ is near $\bar v$ 
(here, the hypothesis that $\nabla \bar f$ and $\nabla \bar g$ are constant is used, 
see the remarks following Theorem \ref{thm: tame inverse}).
This crucial observation is developed in \cite{KoNi} in a more general
framework,
and is presented here in our specific setting in
Theorem \ref{thm: operator A}.
The quadratic part gives then control on
the $L^2(\sP)$-norms of $\partial_i^rF$  and $\partial_i^rG$,
but also on the $L^2(\sP)$-norms of $\partial_1\partial_i^rF$  and
$\partial_1\partial_i^rG$.
Hence the $L^2(\sP)$-norms of $\partial_i^rF$,  $\partial_i^rG$,
$\partial_1\partial_i^rF$  and
$\partial_1\partial_i^rG$ are controlled by the $L^2(\sP)$-norms
of $\partial_i^r\mu$ and  $\partial_i^r\nu$ and by
a small factor times the $H^r(\sP)$-norms of $F$ and $G$.
With all these tools, we get the estimate of Theorem
\ref{thm: general estimate} at the end of Section 
\ref{sec: linearization}, in which the norm of $(f,g)$
in some Sobolev space also appears, the order of which
is under sufficient control.
Although we follow ideas from \cite{KoNi} (see in particular Theorem 2'), explicit estimates
allow one to get explicit regularity results for the solutions
obtained by the Nash-Moser procedure. It may be expected
that these estimates could be improved and thus also the statements
on regularity, but we do not strive in the present work to be optimal.
The lack of compactness mentioned above prevents us from proving $C^\infty$ smoothness of the solution using the method behind  Theorem 2 in \cite{KoNi}.

Our first aim is to find conditions that ensure that $B_{(f,g)}$ 
is positive definite.
In \cite{Bu:2012}, a minimizer of a more general integral functional
could be found in some space of general flows, in a very similar
spirit as in Brenier's work \cite{Brenier}. Hence it could be  expected
that, under appropriate conditions, the quadratic part is 
non-negative at a solution of
\eqref{eq: deux}. In the proof of the following theorem, we also
rely on Poincar\'e's inequality to get the stronger result 
that the quadratic part
is positive definite for $(f,g)$ (not necessarily a solution to \eqref{eq: deux}) sufficiently close to $(\bar f, \bar g)$ and $H''$ sufficiently small (see Theorem \ref{thm: quadratic part}).
For simplicity, we shall assume in the following statement that
\begin{equation}
\label{eq: for simplicity}
|\nabla \bar f|^2+|\nabla \bar g|^2
+\sqrt{(|\nabla \bar f|^2+|\nabla \bar g|^2)^2-4|\bar v|^2}\leq 2,
\quad \bar v \coloneqq \nabla \bar f \times \nabla \bar g.
\end{equation}
As for (small) $\lambda>0$ equation \eqref{eq: deux} remains invariant under the transformation
$$(\widetilde f,\widetilde g)\rightarrow (\lambda\widetilde f,\lambda\widetilde g),~~H\rightarrow \lambda^4 H(\lambda^{-1}\cdot,\lambda^{-1}\cdot),$$
there is no loss of generality.

\begin{thm}
\label{thm: quadratic part}
Assume that $\nabla \bar f$ and $\nabla \bar g$ are constant,
that the first component of $\bar v$ does not vanish and that
\eqref{eq: for simplicity} holds true.
For admissible $(f,g)$ and $(F,G)$,
\begin{equation}
\label{eq: quadra larger}
\begin{aligned}
& B_{(f,g)}((F,G), (F,G))
\\&\geq
\int_{\sP}\Big\{ 
\frac 1 {16}(v\innpr \nabla F)^2+\frac 1{16}(v\innpr \nabla G)^2+(1-O(\|v'\|_{C(\overline \sP)}))\frac {\pi^2\min_{\overline \sP}  v_1^2} {16L^2}(F^2+G^2) 
\\ &\qquad \qquad  \qquad \qquad \qquad
+\partial_f^2 H(f,g)F^2+2\partial_f \partial_g H(f,g)FG+ \partial_g^2 H(f,g) G^2
\Big\}\,dx\,dy\,dz
\end{aligned}
\end{equation}
holds if $(\nabla f,\nabla g)$ is in some small neighborhood of
$(\nabla \bar f,\nabla \bar g)$ in  $C^2(\overline \sP)$
(independent of $H$ admissible).
\end{thm}

\begin{notation}
The notation $u=O(v)$ means that
the norm (or absolute value) of $u$ is less than a constant times $v$ in the relevant domain. 
We also use the notation $u \lesssim v$ to indicate that there exists a constant $C>0$ (independent of $u$ and $v$) such that 
$u\leq Cv$.
\end{notation}

\begin{rem}
It is not essential that $\nabla \bar f$ and $\nabla \bar g$ are constant for this result to hold. The result would still remain true if we instead were  to require that $\rot \bar v=0$  (the other hypotheses remaining the same) and replace 
the coefficient 
$\displaystyle 1-O(\|v'\|_{C(\overline \sP)})$ in \eqref{eq: quadra larger}
by $\displaystyle \operatorname{exp}(-4L\|(v/v_1)'\|_{C(\overline \sP)})$.
This might be useful for considering perturbations of other irrotational flows. See however the remarks following Theorem \ref{thm: tame inverse}.
\end{rem}

\begin{proof}
Under the hypotheses of the theorem, we can assume that the first component of the velocity field
$v=\nabla f\times \nabla g$ never vanishes (like the one of $\bar v$).
We study the various terms separately.

\noindent
{\bf First step.} Let us first show that
\begin{align*}
 &\int_{\sP} 
|\nabla F\times \nabla g+\nabla f\times \nabla G|^2
\,dx\,dy\,dz\\
&\qquad\geq
\int_{\sP}\left\{(v\innpr \nabla F)^2+(v\innpr \nabla G)^2\right\}\,dx\,dy\,dz
\\
&\qquad\geq 
(1-O(\|v'\|_{C(\overline \sP)}))
\frac {\pi^2\min_{\overline \sP}  v_1^2} {L^2}
 \int_{\sP}
(F^2+G^2)\,dx\, dy\,dz
\end{align*}
if $(\nabla f,\nabla g)$ is near enough to $(\nabla \bar f,\nabla \bar g)$
in $C^1(\overline \sP)$.

To this end,  write
$$
\nabla F\times \nabla g+\nabla f\times \nabla G=a\nabla f+b\nabla g+
c \nabla f\times \nabla g.
$$
By taking the scalar product of both sides with 
$\nabla f$, $\nabla g$ and $\nabla f\times \nabla g$
successively, we get
$$\left\{
\begin{array}{l}
(\nabla g\times \nabla f)\innpr \nabla F=a|\nabla f|^2+b \nabla f\innpr\nabla g\\
(\nabla g\times \nabla f)\innpr \nabla G=a \nabla f\innpr\nabla g+b|\nabla g|^2\\
(\nabla g\times(\nabla f\times \nabla g))\innpr \nabla F
+((\nabla f\times \nabla g)\times \nabla f)\innpr \nabla G
=c|\nabla f\times \nabla g|^2
\end{array}\right.$$
and
\begin{align*}
a&=\frac{-|\nabla g|^2(v\innpr \nabla F)+(\nabla f\innpr\nabla g)(v\innpr \nabla G)
}{|\nabla f|^2|\nabla g|^2-\left(\nabla f\innpr\nabla g\right)^2}
\\
&=\frac{-|\nabla g|^2(v\innpr \nabla F)+\left(\nabla f\innpr\nabla g\right)(v\innpr \nabla G)
}{|v|^2}\,,
\\
b&=\frac{-|\nabla f|^2(v\innpr \nabla G)+\left(\nabla f\innpr\nabla g\right)(v\innpr \nabla F)
}{|v|^2}\,,\\
c&=\frac{(v\times \nabla f)\innpr \nabla G+(\nabla g\times v)\innpr \nabla F
}{|v|^2}~.
\end{align*}
Hence
\begin{eqnarray*}
&&\int_{\sP}
 |\nabla F\times \nabla g+\nabla f\times \nabla G|^2
\,dx\,dy\,dz
\\
&&\geq 
 \int_{\sP}|a\nabla f+b\nabla g|^2\,dx\,dy\,dz
\\
&&=
 \int_{\sP}
(a~~b)\left(\begin{array}{c c}|\nabla f|^2&\nabla f\innpr\nabla g
\\\left(\nabla f\innpr\nabla g\right)&|\nabla g|^2\end{array}\right)
\left(\begin{array}{c}a\\b\end{array}\right)
\,dx\,dy\,dz
\\
&&=
\int_{\sP}\frac{1}{|v|^4}
\Big(v\innpr\nabla F~~~v\innpr\nabla G\Big)
\left(\begin{array}{c c}-|\nabla g|^2&\nabla f\innpr\nabla g\\
\nabla f\innpr\nabla g&-|\nabla f|^2\end{array}\right)
\\&&\qquad\times
\left(\begin{array}{c c}|\nabla f|^2&\nabla f\innpr\nabla g
\\\nabla f\innpr\nabla g&|\nabla g|^2\end{array}\right)
\\&&\qquad
\times \left(\begin{array}{c c}-|\nabla g|^2&\nabla f\innpr\nabla g
\\\nabla f\innpr\nabla g&-|\nabla f|^2\end{array}\right)
\left(\begin{array}{c}v\innpr\nabla F\\v\innpr\nabla G\end{array}\right)
\,dx\,dy\,dz
\\
&&=
 \int_{\sP}\frac{1}{|v|^4}
\Big(v\innpr\nabla F~~~v\innpr\nabla G\Big)
\left(\begin{array}{c c}-|\nabla g|^2&\nabla f\innpr\nabla g\\
\nabla f\innpr\nabla g&-|\nabla f|^2\end{array}\right)
\\&&\qquad\times
\left(\begin{array}{c c}-|v|^2&0
\\0&-|v|^2\end{array}\right)
\left(\begin{array}{c}v\innpr\nabla F\\v\innpr\nabla G\end{array}\right)
\,dx\,dy\,dz\\
\\&&=
 \int_{\sP}\frac{1}{|v|^2}
\Big(v\innpr\nabla F~~~v\innpr\nabla G\Big)
\\&&\qquad \times
\left(\begin{array}{c c}|\nabla g|^2&-\nabla f\innpr\nabla g\\
-\nabla f\innpr\nabla g&|\nabla f|^2\end{array}\right)
\left(\begin{array}{c}v\innpr\nabla F\\v\innpr\nabla G\end{array}\right)
\,dx\,dy\,dz\\
\\&&\geq
\int_{\sP}
\frac{|\nabla f|^2+|\nabla g|^2-\sqrt{(|\nabla f|^2+|\nabla g|^2)^2-4|v|^2}}{2|v|^2} \left\{(v\innpr \nabla F)^2+(v\innpr \nabla G)^2\right\}\,dx\,dy\,dz
\end{eqnarray*}
because the eigenvalues of 
$$\left(\begin{array}{c c}|\nabla g|^2&- \nabla f\innpr\nabla g\\
- \nabla f\innpr\nabla g&|\nabla f|^2\end{array}\right)$$
are $\frac 1 2 \left(
|\nabla f|^2+|\nabla g|^2\pm \sqrt{(|\nabla f|^2+|\nabla g|^2)^2-4|v|^2}
\right)$.
By the simplifying assumption \eqref{eq: for simplicity},
$$\int_{\sP}
|\nabla F\times \nabla g+\nabla f\times \nabla G|^2
\,dx\,dy\,dz\geq
\int_{\sP}\left\{(v\innpr \nabla F)^2+(v\innpr \nabla G)^2\right\}\,dx\,dy\,dz$$
if $(\nabla f,\nabla g)$ is near enough to $(\nabla \bar f,\nabla \bar g)$ 
in $C(\overline{\sP})$.

To obtain the second inequality of the first step,
we now use Poincar\'e's inequality in one dimension
by relying on the fact that $F$ and $G$ vanish on $\{0,L\}\times(0,P_1)\times(0,P_2)$,
and then integrate with respect to the two remaining variables.
We use again that the first component of $\bar v$ does not vanish
and that $v$ is in some small neighborhood of $\bar v$,
so that the first component of $v$ does not vanish either.
Given $(\widetilde y,\widetilde z)\in\RR^2$, let 
$\Gamma_{(\widetilde y,\widetilde z)}\colon [0,L] \rightarrow \RR^2$ 
be the function of the variable $\widetilde x\in[0,L]$ satisfying
$$\Gamma_{(\widetilde y,\widetilde z)}'(\widetilde x)=
\frac 1 {v_1(
\widetilde x, \Gamma_{(\widetilde y,\widetilde z)}(\widetilde x)
)}(v_2(
\widetilde x, \Gamma_{(\widetilde y,\widetilde z)}(\widetilde x)
),v_3(
\widetilde x, \Gamma_{(\widetilde y,\widetilde z)}(\widetilde x)
))$$
with the initial condition
$\Gamma_{(\widetilde y,\widetilde z)}(0)=(\widetilde y,\widetilde z)$.
By Theorem 7.2 of Chapter 1 in 
\cite{CoLe} on the regularity of solutions of ODEs,
the map $(\widetilde x,\widetilde y,\widetilde z)\rightarrow 
\Gamma_{(\widetilde y,\widetilde z)}(\widetilde x)$ is of class 
$C^2(\overline P)$.

Moreover the Jacobian determinant of the map
$(\widetilde y,\widetilde z)\rightarrow 
\Gamma_{(\widetilde y,\widetilde z)}(s)$  is given by
$$\text{exp}\int_{0}^{s}\Div_{(y,z)} (v_2/v_1,v_3/v_1)|_{(\widetilde x,\Gamma_{(\widetilde y,\widetilde z)}(\widetilde x))}\,d\widetilde x.$$
Given $\widetilde x\in(0,L)$, we associate to
$(\widetilde x,\widetilde y,\widetilde z)$ the point
$$(x,y,z)=(\widetilde x, \Gamma_{(\widetilde y,\widetilde z)}
(\widetilde x)).$$
Observe that $x=\widetilde x$.
We denote  by $J(\widetilde x,\widetilde y,\widetilde z)$ the Jacobian
determinant  and  obtain
$$J(s,\widetilde y,\widetilde z)
=\text{exp}\int_{0}^{s}\Div_{(y,z)} (v_2/v_1,v_3/v_1)|_{(\widetilde x,\Gamma_{(\widetilde y,\widetilde z)}(\widetilde x))}\,d\widetilde x
=1+O(\|v'\|_{C(\overline \sP)})$$
uniformly in $(s,\widetilde y,\widetilde z)\in \overline \sP$ 
if $v$ is near enough to
$\bar v$ in $C^1(\overline \sP)$.

Setting
$$\widetilde F(\widetilde x,\widetilde y,\widetilde z)=F(x,y,z),
~~\widetilde G(\widetilde x,\widetilde y,\widetilde z)=G(x,y,z),
~~\widetilde v_1(\widetilde x,\widetilde y,\widetilde z)=v_1(x,y,z),$$
we get 
$$\partial_1\widetilde F(\widetilde x,\widetilde y,\widetilde z)
=\frac{d}{d\widetilde x}F
(\widetilde x, \Gamma_{(\widetilde y,\widetilde z)}(\widetilde x))
=
 \nabla F\innpr
\left(\begin{array}{c}1\\v_2/v_1\\v_3/v_1\end{array}\right)
\text{ at }
(\widetilde x, \Gamma_{(\widetilde y,\widetilde z)}(\widetilde x)),
$$
$$
\widetilde v_1  \partial_1 \widetilde F=v\innpr \nabla F  ,~~
\widetilde v_1 \partial_1 \widetilde G=v\innpr \nabla G $$
and
\begin{eqnarray*}
&&
\int_{\sP}
\left\{(v\innpr \nabla F)^2+(v\innpr \nabla G)^2\right\}
\,dx\,dy\,dz\\
\\&&=
 \int_{\sP}
\left\{( \widetilde v_1\partial_1 \widetilde F)^2
+(\widetilde v_1\partial_1 \widetilde G)^2\right\}
J(\widetilde x,\widetilde y,\widetilde z)
d\widetilde x\,d\widetilde y\,d\widetilde z\\
\\&&\geq 
\min_{\overline \sP}( \widetilde v_1^2J)
\int_{(0,P_1)\times(0,P_2)}\left\{\int_0^L
\left\{(\partial_1 \widetilde F)^2
+(\partial_1 \widetilde G)^2\right\} 
d\widetilde x\right\}d\widetilde y\,d\widetilde z\\
\\&&\geq 
\frac {\pi^2\min_{\overline \sP} \widetilde v_1^2J} {L^2} 
\int_{(0,P_1)\times(0,P_2)}\left\{\int_0^L
(\widetilde F^2+\widetilde G^2)d\widetilde x\right\}
d\widetilde y\,d\widetilde z
\\&&\geq 
\frac {\pi^2\min_{\overline \sP} \widetilde v_1^2J} {L^2
\max_{\overline \sP} J} 
 \int_{\sP}
(F^2+G^2)\,dx\, dy\,d z.
\\&&\geq 
(1-O(\|v'\|_{C(\overline \sP)}))
\frac {\pi^2\min_{\overline \sP}  v_1^2} {L^2}
 \int_{\sP}
(F^2+G^2)\,dx\, dy\,d z
\end{eqnarray*}
if $v$ is in some small neighborhood of $\bar v$ in 
$C^1(\overline \sP)$.

\noindent{\bf Second step.}
We now deal with the term
$\int_{\sP}
\left( \nabla f\times \nabla g\right)\innpr\left(\nabla F\times \nabla G\right)
\,dx\,dy\,dz$.
Write
$$\rot v=\alpha v+\beta v\times \nabla f+\gamma \nabla g \times v$$
with 
$$\alpha=\frac{\rot v\innpr v}{|v|^2},~ 
\beta=\frac{\rot v\innpr\nabla g }{|v|^2},~ 
\gamma=\frac{\rot v\innpr\nabla f }{|v|^2}.~ $$
We get
\begin{eqnarray*}
&& \int_{\sP}\left(\nabla f\times \nabla g\right)\innpr\left(\nabla F\times \nabla G\right) \,dx\,dy\,dz
\\&&=\frac12
\int_{\sP}
 v\innpr\rot(F\nabla G-G\nabla F)
\,dx\,dy\,dz\\
&&
=\frac12
\int_{\sP}
\rot v\innpr\left(F\nabla G-G\nabla F\right)
\,dx\,dy\,dz
\end{eqnarray*}
because
\begin{align*}
0&=\int_{\sP} \Div
\left(v\times \left(F\nabla G-G\nabla F\right)\right)\,dx\,dy\,dz
\\
&=
\int_{\sP} \left( \rot v\innpr\left(F\nabla G-G\nabla F\right) -  v\innpr \rot(F\nabla G-G\nabla F)
\right)\,dx\,dy\,dz. 
\end{align*}
Hence
\begin{eqnarray*}
&& \int_{\sP}\left( \nabla f\times \nabla g\right)\innpr\left(\nabla F\times \nabla G\right) \,dx\,dy\,dz
\\&&=
\frac12
\int_{\sP}\
\left( \alpha v+\beta v\times \nabla f+\gamma \nabla g \times v\right)\innpr \left(F\nabla G-G\nabla F\right)
\,dx\,dy\,dz
\\&&=\frac12
\int_{\sP}\Big\{ 
\alpha\left(  \nabla F\times \nabla g+\nabla f\times \nabla G\right)\innpr\left(  G\nabla f-F\nabla g\right)
\\&&\qquad\qquad
+ \left( \beta v\times \nabla f+\gamma \nabla g \times v\right)\innpr\left(F\nabla G-G\nabla F\right)
\Big\}\,dx\,dy\,dz
\\&&\geq
\int_{\sP}\Big\{ 
-\frac 1 8 |\nabla F\times \nabla g+\nabla f\times \nabla G|^2
-\alpha^2(G^2|\nabla f|^2+F^2|\nabla g|^2)
\\&&\qquad\qquad
+\frac12 \left( \beta v\times \nabla f+\gamma \nabla g \times v\right)\innpr\left(F\nabla G-G\nabla F\right)
\Big\}\,dx\,dy\,dz.
\end{eqnarray*}
The (absolute value of the) first term in this expression does not 
create problems because it can be controlled
by one eighth of the term studied in the first step.
Neither does the second term because it can also be controlled
by any fraction of the term studied in the first step 
(as the second term is quadratic in $(F,G)$  and $|\alpha|$
is as small as needed if $\rot v$ is near enough to $\rot \bar v=0$).
The aim of the next step is to deal with the last term.

\noindent{\bf Third step.}
The aim of this step it to get control of the term
\[
\frac12 \int_{\sP} \left( \beta v\times \nabla f+\gamma \nabla g \times v\right)\innpr\left( F\nabla G-G\nabla F\right) \,dx\,dy\,dz.
\]
First, using $\nabla(FG)=G\nabla F+F\nabla G$, we have
\begin{align*}
&\frac12 \int_{\sP} \left( \beta v\times \nabla f\right)\innpr \left(F\nabla G-G\nabla F\right) \,dx\,dy\,dz\\
&=
\frac12\int_{\sP} \left( \beta v\times \nabla f\right)\innpr\nabla(FG) \,dx\,dy\,dz-\int_{\sP} \left( \beta v\times \nabla f\right)\innpr\left(G\nabla F\right) \,dx\,dy\,dz\\
&= -\frac12 \int_{\sP}FG\left( \beta \rot v+\nabla \beta \times v\right)\innpr\nabla f \, \,dx\,dy\, dz
-\int_{\sP} \left( \beta v\times \nabla f\right)\innpr\left(G\nabla F\right) \,dx\,dy\,dz.
\end{align*}
Similarly, we can rewrite
\begin{align*}
&\frac12 \int_{\sP} \left( \gamma \nabla g\times v\right)\innpr\left(F\nabla G-G\nabla F\right) \,dx\,dy\,dz\\
&=-
\frac12\int_{\sP} \left( \gamma  \nabla g\times v\right)\innpr\nabla(FG) \,dx\,dy\,dz+\int_{\sP} \left( \gamma \nabla g\times v\right)\innpr\left(F\nabla G\right) \,dx\,dy\,dz\\
&=
-\frac12\int_{\sP} FG\left( \gamma \rot v+\nabla\gamma \times v\right)\innpr \nabla g \,dx\,dy\,dz+\int_{\sP} \left( \gamma \nabla g\times v\right)\innpr\left(F\nabla G\right) \,dx\,dy\,dz.
\end{align*}
As
$$| -\beta F v\innpr\left(\nabla F\times \nabla g+\nabla f\times \nabla G\right) |\leq
 2\beta^2F^2|v|^2
+\frac 1 8 |\nabla F\times \nabla g+\nabla f\times \nabla G|^2$$
and
\begin{equation}
\label{eq: deals with bdy term}
\begin{aligned}
0&= ~\int_{\sP} \Div\left(v\times 
\left(-\beta\frac{F^2}{2}
\nabla g+\beta FG\nabla f\right)\right)\,dx\,dy\,dz
\\&= ~\int_{\sP}  \rot v\innpr\left(
-\beta\frac{F^2}{2}\nabla g+\beta FG\nabla f\right) \,dx\,dy\,dz
-\int_{\sP} v\innpr 
\rot\left(-\beta\frac{F^2}{2}\nabla g+\beta FG\nabla f
\right) \,dx\,dy\,dz,
\end{aligned}
\end{equation}
we have
\begin{eqnarray*}
&& \int_{\sP} \frac 1 8 
|\nabla F\times \nabla g+\nabla f\times \nabla G|^2
\,dx\,dy\,dz
\\&&\geq \int_{\sP}\Big\{ 
 -\beta F v\innpr\left(\nabla F\times \nabla g+\nabla f\times \nabla G\right)
- 2\beta^2F^2|v|^2
\Big\}\,dx\,dy\,dz
\\&&=\int_{\sP}\Big\{ 
 v\innpr\Big(
\rot\left(-\beta\frac{F^2}{2}\nabla g+\beta F G \nabla f\right)
\\&&\qquad\qquad+\frac{F^2}{2}\nabla \beta\times \nabla g
-FG\nabla\beta\times \nabla f
-\beta G\nabla F\times \nabla f\Big)
-2\beta^2F^2|v|^2
\Big\}\,dx\,dy\,dz
\\&&\stackrel{\eqref{eq: deals with bdy term}}
=\int_{\sP}\Big\{ 
\rot v\innpr \Big(
-\beta\frac{F^2}{2}\nabla g+\beta F G \nabla f\Big)
\\&&\qquad\qquad+\frac{F^2}{2} v\innpr\left(\nabla \beta\times \nabla g\right)
-FG v\innpr\left(\nabla\beta\times \nabla f\right)
\\&&\qquad\qquad\quad-\beta G v\innpr\left(\nabla F\times \nabla f\right)
- 2\beta^2F^2|v|^2
\Big\}\,dx\,dy\,dz
\end{eqnarray*}
and therefore
\begin{align*}
-\int_{\sP}\left( \beta v\times \nabla f\right)\innpr\left(G\nabla F\right) \,dx\,dy\,dz
&=\int_{\sP}\beta G v\innpr\left(\nabla F \times \nabla f\right) \,dx\,dy\,dz\\
&\geq
-\int_{\sP} \frac 1 8 
|\nabla F\times \nabla g+\nabla f\times \nabla G|^2\, \,dx\, dy\, dz
\\&
\quad+\int_{\sP}\Big\{ 
\rot v\innpr\Big(
-\beta\frac{F^2}{2}\nabla g+\beta F G \nabla f\Big)
+\frac{F^2}{2} v\innpr\left(\nabla \beta\times \nabla g\right) \\
&\qquad \qquad \quad
-FG v\innpr\left(\nabla\beta\times \nabla f\right)
- 2\beta^2F^2|v|^2
\Big\}\,dx\,dy\,dz.
\end{align*}
In the previous computations, substitute $f$ and $F$ by $-g$ and $-G$, $g$ and $G$ by $f$ and $F$, and $\beta$ by $\gamma$,  yielding
\begin{align*}
\int_{\sP}\left( \gamma  \nabla g\times v\right)\innpr\left(F\nabla G\right) \,dx\,dy\,dz
&\geq
-\int_{\sP} \frac 1 8 
|\nabla F\times \nabla g+\nabla f\times \nabla G|^2\, \,dx\, dy\, dz
\\&
\quad+\int_{\sP}\Big\{ 
\rot v\innpr\Big(
-\gamma\frac{G^2}{2}\nabla f+\gamma F G \nabla g\Big)
+\frac{G^2}{2} v\innpr\left(\nabla \gamma\times \nabla f\right) \\
&\qquad \qquad \quad
-FGv\innpr\left(\nabla\gamma\times \nabla g\right)
- 2\gamma^2G^2|v|^2
\Big\}\,dx\,dy\,dz.
\end{align*}
Adding the different contributions, we find that
\begin{align*}
&\frac12 \int_{\sP} \left( \beta v\times \nabla f+\gamma \nabla g \times v\right)\innpr\left(F\nabla G-G\nabla F\right) \,dx\,dy\,dz
\\
&\ge
-\int_{\sP} \frac 1 4 
|\nabla F\times \nabla g+\nabla f\times \nabla G|^2\, \,dx\, dy\, dz
\\&
\quad+\int_{\sP}\Big\{ 
\rot v\innpr\Big(
-\beta\frac{F^2}{2}\nabla g+\frac{FG}{2}\beta  \nabla f\Big)
+\frac{F^2}{2}v\innpr\left(\nabla \beta\times \nabla g\right) \\
&\qquad \qquad \quad
-\frac{FG}{2} v\innpr\left(\nabla\beta\times \nabla f\right)
- 2\beta^2F^2|v|^2
\Big\}\,dx\,dy\,dz\\
&\quad+\int_{\sP}\Big\{ 
\rot v\innpr\Big(
-\gamma\frac{G^2}{2}\nabla f+\frac{FG}{2} \gamma \nabla g\Big)
+\frac{G^2}{2} v\innpr\left(\nabla \gamma\times \nabla f\right) \\
&\qquad \qquad \quad
-\frac{FG}{2} v\innpr\left(\nabla\gamma\times \nabla g\right) 
- 2\gamma^2G^2|v|^2
\Big\}\,dx\,dy\,dz.
\end{align*}
All the absolute values of these terms are controlled 
by multiples of the term studied in the first step.
Moreover $|\nabla \beta|$ and $|\nabla \gamma|$ become small 
if $(\nabla f,\nabla g)$ is near enough to $(\nabla \bar f,\nabla \bar g)$
in $C^2(\overline \sP)$.

\noindent{\bf Last step.}
\begin{eqnarray*}
&& \int_{\sP}\Big\{ \frac12
|\nabla F\times \nabla g+\nabla f\times \nabla G|^2+\left( \nabla f\times \nabla g\right)\innpr\left(\nabla F\times \nabla G\right) \Big\}\,dx\,dy\,dz
\\&&\stackrel{\text{step 2}}\geq
 \int_{\sP}\Big\{ \frac38
|\nabla F\times \nabla g+\nabla f\times \nabla G|^2
-\alpha (G^2|\nabla f|^2+F^2 |\nabla g|^2)
\\&&\qquad\qquad\quad + \frac12\left( \beta v\times \nabla f+\gamma \nabla g \times v\right)\innpr\left(F\nabla G-G\nabla F\right)
\Big\}\,dx\,dy\,dz
\\&&\stackrel{\text{step 3}}\geq
 \int_{\sP}\Big\{ \frac18
|\nabla F\times \nabla g+\nabla f\times \nabla G|^2
-\alpha (G^2|\nabla f|^2+F^2|\nabla g|^2) \,dx\,dy\,dz
\\&&\qquad
+\int_{\sP}\Big\{ 
\rot v\innpr\Big(
-\beta\frac{F^2}{2}\nabla g+\frac{FG}{2}\beta  \nabla f\Big)
+\frac{F^2}{2} v\innpr\left(\nabla \beta\times \nabla g\right) \\
&&\qquad \qquad \quad
-\frac{FG}{2} v\innpr\left(\nabla\beta\times \nabla f\right)
- 2\beta^2F^2|v|^2
\Big\}\,dx\,dy\,dz\\
&&\qquad+\int_{\sP}\Big\{ 
 \rot v\innpr\Big(
-\gamma\frac{G^2}{2}\nabla f+\frac{FG}{2} \gamma \nabla g\Big)
+\frac{G^2}{2} v\innpr\left(\nabla \gamma\times \nabla f\right) \\
&&\qquad \qquad \quad
-\frac{FG}{2} v\innpr\left(\nabla\gamma\times \nabla g\right)
- 2\gamma^2G^2|v|^2
\Big\}\,dx\,dy\,dz
\\&&\stackrel{\text{step 1}}\geq
 \int_{\sP} \frac{1}{16}
|\nabla F\times \nabla g+\nabla f\times \nabla G|^2
\,dx\,dy\,dz
\\&&\stackrel{\text{step 1}}\geq
\int_{\sP}\Big\{ 
\frac 1 {32}(v\innpr \nabla F)^2+\frac 1{32}(v\innpr \nabla G)^2
\\&&\qquad\qquad\quad+(1-O(\|v'\|_{C(\overline \sP)}))
\frac {\pi^2\min_{\overline \sP}  v_1^2} {32L^2}
(F^2+G^2)
\Big\}\,dx\,dy\,dz
\end{eqnarray*}
if $(\nabla f,\nabla g)$ is in some small neighborhood of
$(\nabla \bar f,\nabla \bar g)$ in  $C^2(\overline \sP)$ (independent of $H$). 
\end{proof}

Theorem \ref{thm: quadratic part} implies local uniqueness of solutions (existence will be discussed later).

\begin{thm}
\label{thm: uniqueness}
Assume that $(f,g)$ and $(\widetilde f,\widetilde g)$ and  are admissible (see (Ad1)--(Ad2) above), such that
$$
(f(x,y,z),g(x,y,z))=(\widetilde f(x,y,z),\widetilde g(x,y,z)),~\text{for all } (x,y,z)\in \{0,L\}\times \RR^2,
$$
and both $(f,g)$ and $(\widetilde f,\widetilde g)$ are solutions to \eqref{eq: deux}.
In addition let $(\bar f,\bar g)$ be as in Theorem \ref{thm: quadratic part} and $H$ be as in Theorem \ref{thm: main in intro} (but $H_0$ can be assumed of class $C^2$ only).
If $(\nabla f,\nabla g)$  and 
$(\nabla \widetilde f,\nabla \widetilde g)$ are 
in a sufficiently small open convex neighborhood of $(\nabla \bar f,\nabla \bar g)$
in  $C^2(\overline \sP)$ and $\|H_0''\|_{C(\overline \sQ)}$ is sufficiently small, 
then $(f,g)=(\widetilde f,\widetilde g)$ on $[0,L]\times \RR^2$.
\end{thm}

\begin{proof}
If they were not equal, we could consider 
$$(f_\theta,g_\theta)=\theta(\widetilde f,\widetilde g)+(1-\theta)(f,g)$$
for $\theta$ in some slightly larger interval than $[0,1]$.  The map
$$\theta \rightarrow
\int_{\sP}\Big\{ \frac12|\nabla f_\theta\times \nabla g_\theta|^2+H(f_\theta,g_\theta)\Big\}\,dx\,dy\,dz
$$
would be of class $C^2$, its derivative would vanish at
$\theta=0$ and $\theta=1$, and its second derivative would be strictly positive 
on $[0,1]$
(by Theorem \ref{thm: quadratic part}), which is a contradiction.
\end{proof}

\begin{rem}
The proof of Theorem \ref{thm: uniqueness} 
relies on the local convexity
of the functional \eqref{eq: integral functional}.
It is natural to wonder if local convexity may lead to existence too.
Theorem \ref{thm: quadratic part} shows that the quadratic form $B_{(f,g)}((F,G), (F,G))$ is positive definite 
if $(\nabla f,\nabla g)$ is in some small neighborhood of
$(\nabla \bar f,\nabla \bar g)$ in  $C^2(\overline \sP)$  (independent of $H$ as long as $\|H''(f,g)\|_{C(\overline \sP)}$ is sufficiently small).
However, as mentioned above, the quadratic form is not coercive at $(f,g)=(\bar f,\bar g)$.
This feature creates difficulties in getting good a priori bounds on minimizing sequences.
One can hope that they may converge in some weak sense to some kind of weak solution and indeed such kind of results, in a more general setting, are obtained in \cite{Bu:2012}.
One can also wonder if some kind of regularization of the integral functional followed by a limiting process could lead to regular solutions. If this were feasible, it seems likely that it would rely  on a regularity analysis
similar to the one that follows. We leave these considerations 
for further works.
\end{rem}

To implement a Nash-Moser iteration, we introduce for $\epsilon\in[0,1]$ the {\em regularized quadratic form} 
\begin{align*}
(F,G)&\mapsto \int_{\sP}\Big\{ \frac12
|\nabla F\times \nabla g+\nabla f\times \nabla G|^2
+\left( \nabla f\times \nabla g\right)\innpr\left(\nabla F\times \nabla G\right) 
\\
&\qquad
+\frac{\epsilon}{2}(|\nabla F|^2+|\nabla G|^2)
+\frac12 \big(\partial_f^2 H(f,g)F^2+2\partial_f \partial_g H(f,g) FG+ \partial_g^2 H(f,g) G^2\big)\Big\}\,dx\,dy\,dz,
\end{align*}
which is clearly also positive definite if 
 $(\nabla f,\nabla g)$ is in some small neighborhood of
$(\nabla \bar f,\nabla \bar g)$ in  $C^2(\overline \sP)$ and 
 $\|H''(f,g)\|_{C(\overline \sP)}$ is small enough, uniformly in $\epsilon \in [0,1]$, and coercive for a fixed $\epsilon\in (0,1]$. 
Again, the regularized quadratic form can be written $\frac 1 2 B_{(f,g)}^\epsilon((F,G),(F,G))$,
where $B_{(f,g)}^\epsilon$ is the corresponding symmetric bilinear form.

For an admissible $(f,g)\in C^3(\overline D)$ (see (Ad1)--(Ad2) above),  we are
interested in
the map $(\mu,\nu)\mapsto (F,G)$ defined as follows:
\begin{itemize}
\item
$(F,G)\in H^1_{loc}(D)$ 
is admissible in the sense of (Ad'1)--(Ad'3),
\item
$(\mu,\nu)\in L^2_{loc}(D)$ is $(P_1,P_2)$-periodic in $y$ and $z$,
\item
for all $\delta F,\delta G\in H^1_{loc}(D) $ 
that are admissible in the sense of (Ad'1)--(Ad'3)
\begin{equation}\label{weak linear problem}
B_{(f,g)}^\epsilon((F,G),(\delta F,\delta G))=\int_{\sP}
(\mu\delta F+\nu\delta G)\,dx\,dy\,dz.	
\end{equation}
\end{itemize}

If $(f,g)$ is admissible and $(F,G)$ is admissible in $H^2_{loc}(D)$,
\eqref{weak linear problem} is equivalent to the system
\begin{equation}
\label{eq: equivalent system}
\begin{array}{l}
\displaystyle
\mu=-\Div\Big(\nabla g\times(\nabla F\times\nabla g
+\nabla f\times \nabla G)
+\nabla G\times(\nabla f\times \nabla g)\Big)
\\~~~~~
-\epsilon \Delta F
+\partial_f^2 H(f,g)F+\partial_f \partial_g H(f,g) G,
\\
\displaystyle
\nu=-\Div\Big((\nabla F\times\nabla g
+\nabla f\times \nabla G)\times \nabla f
+(\nabla f\times \nabla g)\times \nabla F\Big)
\\~~~~~
-\epsilon \Delta G
+\partial_f \partial_g H(f,g) F
+\partial_g^2 H(f,g)G.
\end{array}
\end{equation}
In particular, if $\epsilon=0$,
then the linear operator related
to $B_{(f,g)}^\epsilon$ is the linearization of \eqref{eq: deux} around $(f,g)$.

Thanks to the fact that the regularized quadratic form is positive definite,
$(F,G)$ is uniquely defined by $(\mu,\nu)$. We leave for later
the issue of the existence of $(F,G)$ and its regularity, as dealt
with in \cite{KoNi}.

\begin{prop}
\label{prop: first estimate} 
Assume that $\nabla \bar f$ and $\nabla \bar g$ are constant,
that the first component of $\bar v$ does not vanish and that
\eqref{eq: for simplicity} holds true.
If $f,g$ (admissible) 
are of class  $C^3(\overline D)$ and 
$H$ (admissible) is of class $C^2(\RR^2)$,
 $(\nabla f,\nabla g)$ is in some small enough neighborhood of
$(\nabla \bar f,\nabla \bar g)$ in  $C^2(\overline \sP)$ and 
 $\|H''(f,g)\|_{C(\overline \sP)}$ is small enough, then
\begin{equation}
 B_{(f,g)}^\epsilon((F,G),(F,G))  
\geq
\int_{\sP}\Big\{ 
\frac 1 {16}(v\innpr \nabla F)^2+\frac 1{16}(v\innpr \nabla G)^2
+
\frac {\pi^2\min_{\overline \sP}  v_1^2} {32L^2}
(F^2+G^2)\Big\}\,dx\,dy\,dz.
\label{eq: Bfg is coercive}
\end{equation}
Moreover
\begin{equation}
\label{eq: FG less munu} 
\|(F,G)\|_{L^2(\sP)} \leq
\frac {32L^2}{\pi^2\min_{\overline \sP}  v_1^2}\,
\|(\mu,\nu)\|_{L^2(\sP)}
\end{equation}
and
\begin{equation*}
\int_{\sP}\Big\{ 
\frac 1 {16}(v\innpr \nabla F)^2+\frac 1{16}(v\innpr \nabla G)^2
\Big\}\,dx\,dy\,dz
\leq
\frac {32L^2}{\pi^2\min_{\overline \sP}  v_1^2}\,
\|(\mu,\nu)\|^2_{L^2(\sP)}
\end{equation*}
for all periodic $(\mu,\nu)\in L^2_{loc}(D)$ 
and all admissible $(F,G)\in H^1_{loc}(D)$
satisfying \eqref{weak linear problem}.
These estimates are uniform in $\epsilon\in[0,1]$.
\end{prop}
\begin{proof}
Assuming  $|v'|$ and $|H''(f,g)|$ small enough (as we can), we get in 
\eqref{eq: quadra larger}
\begin{align*}
&(1-O(\|v'\|_{C(\overline \sP)}))\frac {\pi^2\min_{\overline \sP}  v_1^2}
 {32L^2}(F^2+G^2)+\frac 1 2 \Big(\partial_f^2 H(f,g)F^2+2\partial_f \partial_g H(f,g) FG+\partial_g^2 H(f,g)G^2\Big)
\\
&\quad \geq 
\frac {\pi^2\min_{\overline \sP}  v_1^2} {64L^2}(F^2+G^2)
\end{align*}
and  inequality \eqref{eq: Bfg is coercive} follows from \eqref{eq: quadra larger}.
Applying \eqref{weak linear problem} to $(\delta F,\delta G)=(F,G)$,
\begin{align*}
\int_{\sP}\Big\{ 
\frac 1 {16}(v\innpr \nabla F)^2+\frac 1{16}(v\innpr \nabla G)^2
+
\frac {\pi^2\min_{\overline \sP}  v_1^2} {32L^2}
(F^2+G^2)\Big\}\,dx\,dy\,dz&\leq
 B_{(f,g)}^\epsilon((F,G),(F,G))  
\\
&
\leq \|(\mu,\nu)\|_{L^2(\sP)}
 \|(F,G)\|_{L^2(\sP)}\, ,
\end{align*}
$$
\|(F,G)\|_{L^2(\sP)} \leq
\frac {32L^2}{\pi^2\min_{\overline \sP}  v_1^2}\,
\|(\mu,\nu)\|_{L^2(\sP)}
$$
and
\begin{equation*}
\int_{\sP}\Big\{ 
\frac 1 {16}(v\innpr \nabla F)^2+\frac 1{16}(v\innpr \nabla G)^2
\Big\}\,dx\,dy\,dz
\leq
\frac {32L^2}{\pi^2\min_{\overline \sP}  v_1^2}\,
\|(\mu,\nu)\|^2_{L^2(\sP)}\, .
\end{equation*}
\end{proof}

\begin{prop}
\label{prop: solve second partial in x}
Assume that the first component of $\bar v$ does not vanish and
that $(\nabla f,\nabla g)$ 
is near enough to $(\nabla \bar f,\nabla \bar g)$ in $C^2(\overline \sP)$.
Then  system \eqref{eq: equivalent system} allows one
to express the  partial derivatives
$\partial^2_{11}F$ and $\partial^2_{11}G$ linearly with respect to
$\mu$, $\nu$, the other second-order partial derivatives of $F$ and $G$,
the first-order partial derivatives of $F$ and $G$, and $F$ and $G$.
The coefficients of these two linear expressions are rational functions of
$f',g',f'',g'',H''(f,g), \epsilon$ (without singularities on $\overline D$).
More precisely,
\begin{align*}
\partial_{11}^2F=a_1\mu&+a_2\nu
+a_3\partial_{12}^2F+a_4\partial_{13}^2F
+a_5\partial_{22}^2F+a_6\partial_{23}^2F+a_7\partial_{33}^2F
\\&+a_8\partial_{12}^2G+a_9\partial_{13}^2G
+a_{10}\partial_{22}^2G+a_{11}\partial_{23}^2G+a_{12}\partial_{33}^2G
\\&+a_{13}\partial_1F+a_{14}\partial_2 F+a_{15} \partial_3 F
+a_{16}\partial_1G+a_{17}\partial_2 G+a_{18} \partial_3G
+a_{19}F+a_{20}G,
\end{align*}
where each $a_i$, $1\le i\le 20$, is of the form
$$
a_i=\frac{Q_i}{v_1^2+\epsilon|(\partial_2f,\partial_3f,\partial_2g,\partial_3g)|^2+\epsilon^2},
$$
for some polynomial
$$
Q_{i}=
\begin{cases}
Q_i(f',g',\epsilon), & 1\le i \le 12,\\
Q_i(f'',g''), & 13\le i \le 18,\\
Q_i(H''), & 19\le i \le 20.
\end{cases}
$$
The denominator does not vanish on $\overline D$ because $(\nabla f,\nabla g)$
is supposed near enough to $(\nabla \bar f,\nabla \bar g)$ and $\epsilon\in[0,1]$.
Moreover, for all integers $1\leq i\leq 20$ and $\ell\geq 0$,
$$ \|a_i\|_{C^{\ell}(\overline \sP)}=
\begin{cases}
O\Big(\|(f,g)\|_{C^{\ell+1}(\overline \sP)}+1\Big), & 1\le i\le 12,\\
O\Big(\|(f,g)\|_{C^{\ell+2}(\overline \sP)}+1\Big), & 13\le i \le 18,\\
O\Big(\|H''(f,g)\|_{C^{\ell}(\overline{\sP})}+\|(f,g)\|_{C^{\ell+1}(\overline \sP)}+1\Big), & 19\le i \le 20.
\end{cases}
$$
if all norms are well defined.
Analogous results hold for $\partial_{11}^2G$ and all
the estimates are uniform in $\epsilon\in[0,1]$.
\end{prop}
\begin{proof}
If we keep only the second-order terms in $(F,G)$, we get
\begin{align*}
\mu&=\nabla g\innpr \rot(\nabla F\times\nabla g
+\nabla f\times \nabla G)
-\epsilon\Delta F+\ldots,\\
\nu&=-\rot(\nabla F\times\nabla g
+\nabla f\times \nabla G)\innpr \nabla f-\epsilon\Delta G+\ldots
\end{align*}
Observe that
\[
\rot(\nabla F\times \nabla g)=\Delta g \nabla F-\Delta F \nabla g+F''\nabla g-g''\nabla F
\]
and thus
\begin{align*}
\mu&=
 \nabla g\innpr\left((F''-\Delta F I)\nabla g\right)
- \nabla g\innpr \left((G''-\Delta GI)\nabla f\right)
-\epsilon \Delta F+\ldots,\\
\nu&=
-\nabla f\innpr\left((F''-\Delta FI)\nabla g\right)
+ \nabla f\innpr\left((G''-\Delta GI)\nabla f\right) -\epsilon\Delta G
+\ldots
\end{align*}
where $I$ is the identity matrix. To see that this allows one 
to express $\partial^2_{11}F$ and $\partial^2_{11}G$ with respect to
$\mu$, $\nu$, the other second-order partial derivatives of $F$ and $G$,
and the first-order partial derivatives of $F$ and $G$, and $F$ and $G$,
it is sufficient 
to study
$$
\begin{array}{l}
\mu=-\partial_{11}^2F\nabla g\innpr\left(J \nabla g\right)
+\partial_{11}^2G \nabla f\innpr\left(J\nabla g\right)
-\epsilon\partial_{11}^2F+\ldots
\\
\nu=
\partial_{11}^2F \nabla f\innpr\left(J\nabla g\right)
-\partial_{11}^2G \nabla f\innpr\left(J\nabla f\right) -\epsilon \partial_{11}^2G
+\ldots
\end{array}
$$
where $J$ is the diagonal matrix with entries $(0,1,1)$ on the diagonal and the remainders now also contain the other second-order partial derivatives of $F$ and $G$.
The discriminant of this system for $(\partial_{11}^2F,\partial_{11}^2G)$
is
\begin{align*}
(|J\nabla g|^2+\epsilon)(|J\nabla f|^2+\epsilon)
-\left(\left( J\nabla f\right)\innpr\left(J\nabla g\right)\right) ^2
&=|(J\nabla f)\times (J\nabla g)|^2
+\epsilon|J\nabla f|^2+\epsilon|J\nabla g|^2+\epsilon^2 
\\&=v_1^2
+\epsilon|J\nabla f|^2+\epsilon|J\nabla g|^2+\epsilon^2.
\end{align*}

We estimate 
$\|a_i\|_{C^{\ell}(\overline \sP)}$, $1\le i \le 20$ using the inequality
\begin{equation}
\label{eq: composition}
\|\xi(u_1, \ldots, u_N)\|_{C^k(\overline \sP)} \le C\|\xi\|_{C^k} \big(1+\|u_1\|_{C^k(\overline \sP)}+ \cdots + \|u_N\|_{C^k(\overline \sP)}\big), 
\end{equation}
for $\xi\in C^k([-M, M]^N)$ and $u_j \in C^k(\overline \sP)$ with $\|u_j\|_{C(\overline \sP)} \le M$ for $1\le j\le N$, which e.g.~follows by interpolation in $C^k$ spaces (see e.g.\ Theorem 2.2.1 on p.\ 143 of \cite{Ha}) and the Fa\`a di Bruno formula.
Hence
$$O\left( \|a_i\|_{C^{\ell}(\overline \sP)}\right)=
\begin{cases}
O(\|(f,g)\|_{C^{\ell+1}(\overline \sP)}+1),& 1\le i \le 12,\\
O(\|(f,g)\|_{C^{\ell+2}(\overline \sP)}+1),& 13\le i \le 18,\\
O(\|H''(f,g)\|_{C^\ell(\overline \sP)}+\|(f,g)\|_{C^{\ell+1}(\overline \sP)}+1), & 19\le i \le 20.
\end{cases}
$$
\end{proof}

We now study to which extent $B_{(f,g)}^\epsilon$ 
commutes with differentiations in $y$ and
$z$, following the general approach of \cite{KoNi}. 

\begin{thm}
\label{thm: operator A}
Let $(\nabla f,\nabla g)$ be in any bounded subset of $C^1(\overline\sP)$,
$r\in\{1,2,3,\ldots\}$, 
$(f,g)\in C^{r+2}(\overline D)$, $H\in C^{r+2}(\RR^2)$  
and $(F,G)\in H^{2r+1}_{loc}(D)$ (all admissible).
Then, for $j\in \{2,3\}$,
\begin{align*}
&B_{(f,g)}^\epsilon((\partial_j^{r}F,\partial_j^{r}G),(\partial_j^{r}F,\partial_j^{r}G) )
-B_{(f,g)}^\epsilon((F,G),(-1)^r (\partial_j^{2r}F,\partial_j^{2r}G) )
\\
&=
\sum_{p\in \sS}
\int_{\sP}\partial_j^{2r-s_p-t_p}
L_{p}\, \partial_j^{s_p}u_p\, \partial_j^{t_p}v_p\, \,dx\,dy\,dz
+\sum_{p\in \widetilde \sS}
\int_{\sP}\partial_j^{r-\widetilde s_p}
\widetilde L_{p}\, \partial_j^{\widetilde s_p}\widetilde u_p\, \partial_j^{r}\widetilde v_p\, \,dx\,dy\,dz,
\end{align*}
where, for each $p$ in some finite sets $\sS$ and $\widetilde \sS$ of indices,
$$
0\leq s_p\leq t_p\leq r-1,~2\leq  2r-s_p-t_p\leq r+1, \qquad 0\leq \widetilde s_p\leq r-1
$$
and
$$
\{u_p,v_p\}\subset \{\partial_1 F,\partial_2F,\partial_3F,\partial_1G, \partial_2G,\partial_3G\},
\qquad
\{\widetilde u_p,\widetilde v_p\}\subset \{F,G\}.
$$
For each $p$, the coefficient 
$L_p(x,y,z)$ 
is a polynomial of all partial derivatives of $f$ and $g$ of order $1$, while 
$\widetilde L_p$ is a second order partial derivative of $H$ (with respect to $f$ and $g$).
Moreover we have the following estimate, where the dependence on $r$
is more explicitly stated: 
\begin{equation}
\label{eq: r dependence}
\left\|\sum_{\substack{p\in \sS:\\\, s_p=t_p=r-1}} \partial_j^{2r-s_p-t_p}L_p\right\|_{C(\overline \sP)}
=\left\|\sum_{\substack{p\in \sS:\\\, s_p=t_p=r-1}} \partial_j^{2}L_p\right\|_{C(\overline \sP)}= O(r^2)
\|(\partial_j\nabla f,\partial_j\nabla g)
\|_{C^1(\overline \sP)}
\end{equation}
(the function $O(r^2)$ being independent of $f,g,F,G, H''(f,g)$ and $\epsilon$).
Finally, for the other indices $p$,
\begin{equation}
\label{eq: sup L}
\begin{aligned}
&\|\partial_{j}^{2r-s_p-t_p}L_p\|_{C(\overline\sP)}
= O\Big(\|(\nabla f,\nabla g)\|_{C^{2r-s_p-t_p}(\overline \sP)}+1\Big), && p \in \sS, 
\\
&\|\partial_j^{r-\tilde s_p} \widetilde L_p\|_{C(\overline \sP)}
= O\Big(\|H''(f,g)\|_{C^{r-\widetilde s_p}(\overline \sP)}\Big), && p\in \widetilde \sS,
\end{aligned}
\end{equation}
where the constants in the estimates may depend on $r$.
\end{thm}

\begin{rems}
The expression
$$
B_{(f,g)}^\epsilon((\partial_j^{r}F,\partial_j^{r}G),(\partial_j^{r}F,\partial_j^{r}G) )
-B_{(f,g)}^\epsilon((F,G),(-1)^r (\partial_j^{2r}F,\partial_j^{2r}G) ),
$$
would vanish if $(\nabla f,\nabla g)$ and
$H''(f,g)$ were  independent of $y$ and $z$,
and the statement 
allows one  to estimate its size otherwise.
In the statement, we add the property $s_p\leq t_p$. 
In fact we shall
omit this property in the proof, as it is easy to get it by renaming
$s_p$ and $t_p$.
The statement would be much easier if we would 
aim at the weaker inequality $0\leq s_p\leq t_p\leq r$ (the proof would then rely on straightforward
integrations by parts). 
The crucial regularity gain $s_p,t_p\leq r-1$
has been explored in a general setting in \cite{KoNi}.
\end{rems}

\begin{proof}
The typical term of
$B_{(f,g)}^\epsilon((F,G),(F,G))$
is of either of the form 
$$
\int_{\sP}2L(x,y,z)u(x,y,z)v(x,y,z)\,dx\,dy\,dz,
$$
where 
$$\{u,v\}\subset \{\partial_1 F,\partial_2F,\partial_3F,\partial_1 G,\partial_2G,\partial_3G\}$$
and the coefficient 
$L(x,y,z)$  can be expressed as a polynomial of	
the partial derivatives of $f$ and $g$ of order $1$, or 
of the form
$$
\int_{\sP}2\widetilde L(x,y,z)\widetilde u(x,y,z)\widetilde v(x,y,z)\,dx\,dy\,dz,
$$
where
$$\{\widetilde u,\widetilde v\}\subset \{F,G\}$$ 
and $\widetilde L$ is equal to $\partial_f^2H(f,g)$, $2\partial_f \partial_g H(f,g)$ or $\partial_g^2H(f,g)$.
The typical term of
$$
B_{(f,g)}^\epsilon((\partial_j^{r}F,\partial_j^{r}G),(\partial_j^{r}F,\partial_j^{r}G) )
-B_{(f,g)}^\epsilon((F,G),(-1)^r (\partial_j^{2r}F,\partial_j^{2r}G) )
$$
is therefore either of the form 
\begin{equation*}
\int_{\sP}\Big(
2L\partial_j^ru\partial_j^rv
-(-1)^rLv\partial_j^{2r} u-(-1)^rLu\partial_j^{2r}v
\Big)\,dx\,dy\,dz
\end{equation*}
or
\begin{equation*}
\int_{\sP}\Big(
2\widetilde L\partial_j^r \widetilde u\partial_j^r \widetilde v
-(-1)^r \widetilde L\widetilde v\partial_j^{2r} \widetilde u-(-1)^r\widetilde L\widetilde u\partial_j^{2r}\widetilde v
\Big)\,dx\,dy\,dz.
\end{equation*}
We only give the details for the first type of term since the argument for the second is similar but simpler (move $r$ derivatives using integration by parts).

We get as in \cite{KoNi} (but in a simpler setting)
\begin{align*}
\int_{\sP}-(-1)^r Lv\partial_j^{2r} u\,dx\,dy\,dz
&=\int_{\sP} \partial_j^{r+1}(Lv)\partial_j^{r-1} u\,dx\,dy\,dz
\\&=\int_{\sP}\sum_{k=0}^{r+1}\binom{r+1}{k}
\partial_j^{r+1-k}L\partial_j^{k}v\partial_j^{r-1} u\,dx\,dy\,dz\\
&
=\int_{\sP}L\partial_j^{r+1}v\partial_j^{r-1} u\,dx\,dy\,dz
+\int_{\sP}(r+1)\partial_jL\partial_j^{r}v\partial_j^{r-1} u\,dx\,dy\,dz\\
&
\quad+\int_{\sP}\frac 1 2 r(r+1)\partial_j^2L\partial_j^{r-1}v\partial_j^{r-1} u\,dx\,dy\,dz
\\&\quad+\int_{\sP}\sum_{k=0}^{r-2}\binom{r+1}{k}
\partial_j^{r+1-k}L\partial_j^{k}v\partial_j^{r-1} u\,dx\,dy\,dz
\end{align*}
and thus, together with the equality one gets by permuting $u$ and $v$,
\begin{align*}
&\int_{\sP}\Big(
2L\partial_j^ru\partial_j^rv
-(-1)^r Lv\partial_j^{2r} u-(-1)^rLu\partial_j^{2r}v
\Big)\,dx\,dy\,dz\\
&
=\int_{\sP}L\partial_j^2\left(\partial_j^{r-1}u\partial_j^{r-1}v\right) \,dx\,dy\,dz
\\&\quad+\int_{\sP}(r+1)\partial_jL\partial_j\left(\partial_j^{r-1}u\partial_j^{r-1} v\right)\,dx\,dy\,dz
+\int_{\sP}r(r+1)\partial_j^2L\partial_j^{r-1}u\partial_j^{r-1} v\,dx\,dy\,dz
\\&\quad+\int_{\sP}\sum_{k=0}^{r-2}\binom{r+1}{k}
\partial_j^{r+1-k}L\left(\partial_j^{k}v\partial_j^{r-1}u+\partial_j^ku\partial_j^{r-1}v\right) \,dx\,dy\,dz\\
&
=r^2\int_{\sP}\partial_j^2L\partial_j^{r-1}u\partial_j^{r-1} v\,dx\,dy\,dz
+\int_{\sP}\sum_{k=0}^{r-2}\binom{r+1}{k}
\partial_j^{r+1-k}L\left(\partial_j^{k}v\partial_j^{r-1}u+\partial_j^ku\partial_j^{r-1}v\right) \,dx\,dy\,dz.
\end{align*}
With respect to  the $j$-th variable,
$L$ is differentiated at most $r+1$ times,
and $u$ and $v$ at most $r-1$ times.
Moreover
the term containing $\partial_j^{r-1}u\partial_j^{r-1}v$  is given by
\begin{equation*}
r^2\int_{\sP}
\partial_j^2L\partial_j^{r-1}u\partial_j^{r-1}v\,dx\,dy\,dz,
\end{equation*}
where
$$\|\partial_j^2L\|_{C(\overline\sP)}
=O\Big(
\|(\partial_j\nabla f,\partial_j\nabla g)
\|_{C^1(\overline \sP)}\Big)$$ 
(using the fact that $(\nabla f,\nabla g)$ is supposed 
to be in some bounded subset of the algebra $C^1(\overline\sP)$).
To get  \eqref{eq: sup L}, we use \eqref{eq: composition} with $k=2r-s_p-t_p$ and $\xi=L$.
\end{proof}

In the two following results, everything is uniform in $\epsilon\in[0,1]$ and we do not state explicitly
the dependence on $\epsilon$.
\begin{prop}
\label{prop: second estimate}
If $(f,g,H)\in C^3(\overline D)\times C^3(\overline D)\times C^3(\RR^2)$ 
is admissible,
$(\nabla f,\nabla g)$ is in some small enough neighborhood of
$(\nabla \bar f,\nabla \bar g)$ in  $C^2(\overline \sP)$ and
$\|H''(f, g)\|_{C(\overline \sP)}$ is small enough, then
\begin{equation}
\label{eq: FG1 less munu1}
\|(F,G)\|_{H^1(\sP)}=
O\Big(\|H''(f,g)\|_{C^1(\overline{\sP})} +1 \Big)\|(\mu,\nu)\|_{H^1(\sP)}
\end{equation}
and
$$\sum_{j\in\{2,3\}}\int_{\sP}\Big\{ 
\frac 1 {16}(v\innpr \nabla \partial_jF)^2
+\frac 1{16}(v\innpr \nabla \partial_jG)^2
\Big\}\,dx\,dy\,dz
=
O\Big(\|H''(f,g)\|_{C^1(\overline{\sP})} +1 \Big)^2 \|(\mu,\nu)\|_{H^1(\sP)}^2
$$
for all periodic $(\mu,\nu)\in H^1_{loc}(D)$ 
and all admissible $(F,G)\in H^3_{loc}(D)$
satisfying \eqref{weak linear problem}.
\end{prop}
\begin{proof}
In Theorem \ref{thm: operator A}, we consider $r=1$.
Applying \eqref{weak linear problem} to $(\delta F,\delta G)
=-(\partial_j^2 F,\partial_j^2G)$ with $j\in\{2,3\}$ and using 
Proposition \ref{prop: first estimate}, we get
\begin{align*}
&
\int_{\sP}\Big\{ 
\frac 1 {16}(v\innpr \nabla \partial_jF)^2
+\frac 1{16}(v\innpr \nabla \partial_jG)^2
+
\frac {\pi^2\min_{\overline \sP}  v_1^2} {32L^2}
((\partial_jF)^2+(\partial_jG)^2)\Big\}\,dx\,dy\,dz
\\&
\leq
B_{(f,g)}^\epsilon((\partial_jF,\partial_jG),(\partial_j F,\partial_j G))\\
&= B_{(f,g)}^\epsilon((F,G),-(\partial_j^2 F,\partial_j^2 G))\\
&\qquad
+\Big\{B_{(f,g)}^\epsilon((\partial_jF,\partial_jG),(\partial_j F,\partial_j G))
-B_{(f,g)}^\epsilon((F,G),-(\partial_j^2 F,\partial_j^2 G))\Big\}
\\&
\stackrel{\mathclap{\eqref{weak linear problem}}}
{=}
\int_{\sP}(\partial_j\mu\partial_j F
+\partial_j\nu\partial_j G)\,dx\,dy\,dz
\\
&\qquad+\Big\{B_{(f,g)}^\epsilon((\partial_jF,\partial_jG),(\partial_j F,\partial_j G))-B_{(f,g)}^\epsilon((F,G),-(\partial_j^2 F,\partial_j^2 G))\Big\}
\\& \stackrel{
\eqref{eq: r dependence},\eqref{eq: sup L}
}\leq
\|(\partial_j \mu,\partial_j\nu)\|_{L^2(\sP)}
 \|(\partial_jF,\partial_jG)\|_{L^2(\sP)}
+O\Big(
\|(\partial_j\nabla f,\partial_j\nabla g)
\|_{C^1(\overline \sP)}
\Big)
\|(F,G)\|_{H^1(\sP)}^2\\
& \qquad+O\Big(
\|H''(f,g)\|_{C^1(\overline \sP)}
\Big)
\|(F,G)\|_{L^2(\sP)}
\|(\partial_j F,\partial_j G)\|_{L^2(\sP)}
\\& \leq
\|(\partial_j \mu,\partial_j\nu)\|_{L^2(\sP)}
 \|(\partial_jF,\partial_jG)\|_{L^2(\sP)}
+O\Big(
\|(\partial_j\nabla f,\partial_j\nabla g)
\|_{C^1(\overline \sP)}
\Big)
\|(F,G)\|_{H^1(\sP)}^2\\
& \qquad+\delta^{-1}O\Big(
\|H''(f,g)\|_{C^1(\overline \sP)}
\Big)^2
\|(F,G)\|_{L^2(\sP)}^2+
\delta \|(\partial_j F,\partial_j G)\|_{L^2(\sP)}^2
~.
\end{align*}

If, in addition, 
$$
\|(\partial_2\nabla f,\partial_3 \nabla f,\partial_2\nabla g,\partial_3 \nabla g)
\|_{C^1(\overline \sP)}< \delta
$$
and $\delta>0$ is small enough, we get 
(note that the coefficient $32$ is replaced by $64$, and later by $128$)
\begin{align*}
&
\sum_{j\in\{2,3\}}\int_{\sP}\Big\{ 
\frac 1 {16}(v\innpr \nabla \partial_jF)^2
+\frac 1{16}(v\innpr \nabla \partial_jG)^2
+
\frac {\pi^2\min_{\overline \sP}  v_1^2} {64L^2}
((\partial_jF)^2+(\partial_jG)^2)\Big\}\,dx\,dy\,dz
\\&\lesssim
\|(\mu, \nu)\|_{H^1(\sP)}^2
+\delta^{-1}(\|H''(f,g)\|_{C^1(\overline \sP)}+1)^2
\|(F,G)\|_{L^2(\sP)}^2
+\delta \|(\partial_1F,\partial_1G)\|_{L^2(\sP)}^2
\\&\stackrel{\mathclap{\eqref{eq: FG less munu}}}{\lesssim}
\|(\mu, \nu)\|_{H^1(\sP)}^2
+\delta^{-1}(\|H''(f,g)\|_{C^1(\overline \sP)}+1)^2
\|(\mu,\nu)\|_{L^2(\sP)}^2
+\delta \|(\partial_1F,\partial_1G)\|_{L^2(\sP)}^2.
\end{align*}

Using the last inequality
in Proposition \ref{prop: first estimate} 
to estimate $\|\partial_1F\|_{L^2(\sP)}^2$
and $\|\partial_1G\|_{L^2(\sP)}^2$
(using also the fact that the first component of $v$
never vanishes), we obtain
$$\|(\partial_1 F,\partial_1 G)\|^2_{L^2(\sP)}
=O\Big(\|(\mu,\nu,\partial_2F,\partial_2G,\partial_3F,\partial_3G)\|^2_{L^2(\sP)}\Big)$$
and
\begin{align*}
&\sum_{j\in\{2,3\}}\int_{\sP}\Big\{ 
\frac 1 {16}(v\innpr \nabla \partial_jF)^2
+\frac 1{16}(v\innpr \nabla \partial_jG)^2
\Big\}\,dx\,dy\,dz+
\frac {\pi^2\min_{\overline \sP}  v_1^2} {128L^2}
\|(\nabla F,\nabla G)\|^2_{L^2(\sP)}\\
&=
O\Big(\|H''(f,g)\|_{C^1(\overline{\sP})} +1 \Big)^2 \|(\mu,\nu)\|_{H^1(\sP)}^2.
\end{align*}
We get \eqref{eq: FG1 less munu1} by combining this with \eqref{eq: FG less munu}.
\end{proof}

By induction, we get the following theorem.
\begin{thm}
\label{thm: tame inverse}
Let $r\geq 1$ be an integer,
$(f,g)\in H^{r+4}_{loc}(D)$ (admissible)
be in some small enough neighborhood of $(\bar f,\bar g)$ in  
$H^5(\sP)$,
$H\in C^{2}(\RR^2)$ be admissible, $H''(f,g)\in C^{r}(\overline{\sP})$ and
$H''(f,g)$ be small enough in $C(\overline{\sP})$.
There exists a constant $C_r>0$ such that, if
\begin{equation}
\label{eq: with partial derivatives 2 and 3, recurrence}
\|(\partial_2\nabla f,\partial_3 \nabla f,\partial_2\nabla g,\partial_3 \nabla g)
\|_{C^1(\overline \sP)}
< C_r^{-1},
\end{equation}
then
\begin{equation}
\label{eq: tame inverse}
\begin{aligned}
&\sum_{j\in\{2,3\}}\int_{\sP}\Big\{ 
\frac 1 {16}(v\innpr \nabla \partial^r_jF)^2
+\frac 1{16}(v\innpr \nabla \partial^r_jG)^2
\Big\}\,dx\,dy\,dz+\|(F,G)\|_{H^r(\sP)}^2\\
&\leq 
C_r \|(\mu,\nu)\|_{H^r(\sP)}^2
+C_r\|(\mu,\nu)\|_{H^1(\sP)}^2
 \Big(\|(f,g)\|_{H^{r+4}(\sP)}
+\|H''(f,g)\|_{C^{r}(\overline \sP)}+1
\Big)^2
\end{aligned}
\end{equation}
for all periodic $(\mu,\nu)\in H^r(\sP)$ 
and all admissible $(F,G)\in H^{2r+1}_{loc}(D)$
satisfying \eqref{weak linear problem}.
\end{thm}

\begin{rems}$ $
\begin{itemize}
\item In \eqref{eq: with partial derivatives 2 and 3, recurrence},
all terms in the norm are differentiated at least once with respect to 
$y$ or $z$. 
In the first sentence of the statement,
the small neighborhood and the small bound on the size of $H''(f,g)$
in $C(\overline{\sP})$
are independent 
of $r\geq 1$. The constant $C_r$ can depend on them, on $r$, 
$\overline f$ and
$\overline g$, but not on $H$, $f$ and $g$.  

\item The $r$ dependence in \eqref{eq: with partial derivatives 2 and 3, recurrence}
is due to the appearance of $r$ in the estimate \eqref{eq: r dependence} in Theorem \ref{thm: operator A} (see also \eqref{eq: r dependence 2} below).

\item Unlike Theorem \ref{thm: quadratic part} where the constancy of $\bar v$ was not essential it really does matter here (see \eqref{eq: with partial derivatives 2 and 3, recurrence}).
\end{itemize}
\end{rems}

\begin{proof}
As the result is already known for $r=1$ (see Proposition
\ref{prop: second estimate})
let us assume that $r\geq 2$.

\noindent
{\bf First step.}
We first bound from above
$$\int_{\sP}\Big\{ 
\frac 1 {16}(v\innpr \nabla \partial^r_jF)^2
+\frac 1{16}(v\innpr \nabla \partial^r_jG)^2
+
\frac {\pi^2\min_{\overline \sP}  v_1^2} {32L^2}
((\partial_j^rF)^2+(\partial_j^rG)^2)\Big\}\,dx\,dy\,dz
$$
for $j\in\{2,3\}$.
We shall deal with $\partial_1^r F$ and $\partial_1^r G$ in the third
and fourth steps.
Applying \eqref{weak linear problem} to $(\delta F,\delta G)
=(-1)^r(\partial_j^{2r} F,\partial_j^{2r}G)$ with $j\in\{2,3\}$, and using Proposition \ref{prop: first estimate} we get
\begin{align*}
&
\int_{\sP}\Big\{ 
\frac 1 {16}(v\innpr \nabla \partial^r_jF)^2
+\frac 1{16}(v\innpr \nabla \partial^r_jG)^2+
\frac {\pi^2\min_{\overline \sP}  v_1^2} {32L^2}
((\partial_j^rF)^2+(\partial_j^rG)^2)\Big\}\,dx\,dy\,dz
\\
&\leq
B_{(f,g)}^\epsilon((\partial_j^rF,\partial_j^rG),(\partial_j^r F,\partial_j^r G))
\\&
=B_{(f,g)}^\epsilon((F,G),(-1)^r(\partial_j^{2r} F,\partial_j^{2r} G))
\\&\qquad+\Big\{B_{(f,g)}^\epsilon((\partial_j^rF,\partial_j^rG),(\partial_j^r F,\partial_j^r G))
-B_{(f,g)}^\epsilon((F,G),(-1)^r(\partial_j^{2r} F,\partial_j^{2r} G))\Big\}
\\&\stackrel{\mathclap{\eqref{weak linear problem}}}=
\int_{\sP}(\partial_j^r\mu\partial_j^r F
+\partial_j^r\nu\partial_j^rG)\,dx\,dy\,dz
\\&\qquad+\Big\{B_{(f,g)}^\epsilon((\partial_j^rF,\partial_j^rG),(\partial_j^r F,\partial_j^r G))
-B_{(f,g)}^\epsilon((F,G),(-1)^r(\partial_j^{2r} F,\partial_j^{2r} G))\Big\}.
\end{align*}
By Theorem \ref{thm: operator A},
\begin{equation}
\label{eq: formula with sum} 
\begin{aligned}
&\int_{\sP}\Big\{ 
\frac 1 {16}(v\innpr \nabla \partial^r_jF)^2
+\frac 1{16}(v\innpr \nabla \partial^r_jG)^2+
\frac {\pi^2\min_{\overline \sP}  v_1^2} {32L^2}
((\partial_j^rF)^2+(\partial_j^rG)^2)\Big\}\,dx\,dy\,dz
\\
&\leq
\|(\partial_j^r \mu,\partial_j^r\nu)\|_{L^2(\sP)}
 \|(\partial_j^rF,\partial_j^rG)\|_{L^2(\sP)}\\
 &\qquad
+O(r^2)
\|(\partial_j\nabla f,\partial_j\nabla g)\|_{C^1(\overline \sP)}
\|(F,G)\|_{H^r(\sP)}^2
\\
&\qquad+\sum O\Big(\|(f,g)\|
_{H^{k_1+3}(\sP)}+1\Big) \|(F,G)\|_{H^{k_2+1}(\sP)}
\|(F,G)\|_{H^{k_3+1}(\sP)}\\
&\qquad 
+\sum O\Big( \|H''(f,g)\|_{C^{r-k_4}(\overline{\sP})}\Big)
\|(F,G)\|_{H^{k_4}(\sP)}
\|(F,G)\|_{H^{r}(\sP)}
\end{aligned}
\end{equation}
where the sums are over all integers $k_1,k_2,k_3\geq 0$ such that
$$k_1+k_2+k_3=2r,\ k_1\leq r+1,\ k_2\leq k_3\leq r-1,\, 
k_2+k_3< 2r-2$$
(this implies $k_1>2$ and, as $r\geq 2$,  $k_2+k_3>0$)
and $0\le k_4\le r-1$. Here and in the following estimates, we only indicate the $r$ dependence in the coefficients 
of $\|(F,G)\|_{H^r(\sP)}$. We don't keep track of the $r$ dependence of the lower order terms.

By standard interpolation in Sobolev spaces based on the equality
$k_j+1=\frac{r-1-k_j}{r-1}\cdot 1+\frac{k_j}{r-1}\cdot r$, $j=2,3$,
(see e.g.~section 4.3 in \cite{Han-Hong}), 
the first sum can be estimated by
\begin{eqnarray*}
&&\sum O\Big(\|(f,g)\|
_{H^{k_1+3}(\sP)}
+1\Big)
\|(F,G)\|_{H^{1}(\sP)}^{\frac{k_1-2}{r-1}}
\|(F,G)\|_{H^{r}(\sP)}^{\frac{2r-k_1}{r-1}}
\\&&=\sum\left\{  O\Big(\|(f,g)\|
_{H^{k_1+3}(\sP)}
+1\Big)^{\frac{2(r-1)}{k_1-2}}\delta^{\,-\frac{2r-k_1}{k_1-2}}\|(F,G)\|^2_{H^{1}(\sP)}\right\}^{\frac{k_1-2}{2(r-1)}}
\left\{\delta\|(F,G)\|^2_{H^{r}(\sP)}\right\}^{\frac{2r-k_1}{2(r-1)}},
\end{eqnarray*}
where $\delta>0$ will be chosen as small as needed.
The choice of $\delta>0$ can depend on $r$, 
$\overline f$ and
$\overline g$, but not on $(F,G)$, $(\mu,\nu)$, $H$, $f$ and $g$.  
In what follows, we write explicitly some negative powers of $\delta$,
even when they can be merged with other positive factors, for example
those referred to in the notation $\lesssim$
(possibly depending on $r$, 
$\overline f$ and $\overline g$).  
By Young's inequality for products, $xy\le p^{-1}x^p+q^{-1}y^q$ with $p=2(r-1)/(k_1-2)$, $q=2(r-1)/(2r-k_1)$,
and interpolation based on the equality
\[
k_1+3=\frac{r+1-k_1}{r-1}\cdot 5+\frac{k_1-2}{r-1}\cdot (r+4),
\]
this can in turn be estimated by
\begin{align*}
& \delta \|(F,G)\|_{H^{r}(\sP)}^{2}
+\sum 
\delta^{\,-\frac{2r-k_1}{k_1-2}}
O\Big(\|(f,g)\|
_{H^{k_1+3}(\sP)}
+1\Big)^{\frac{2(r-1)}{k_1-2}}
\|(F,G)\|_{H^{1}(\sP)}^{2}\\
&\lesssim 
\delta\|(F,G)\|_{H^{r}(\sP)}^{2}+
\sum
\delta^{\,-\frac{2r-k_1}{k_1-2}}
 \Big(\|(f,g)\|_{H^{5}(\sP)}
+1\Big)^{\frac{2(r+1-k_1)}{k_1-2}} \Big(\|(f,g)\|_{H^{r+4}(\sP)}
+1\Big)^{2}
\|(F,G)\|_{H^{1}(\sP)}^{2}.
\end{align*}
By Proposition \ref{prop: second estimate}, the sum is thus estimated above:
\begin{equation}
\begin{aligned}
\label{eq: estimate sum}
&\sum \Big(\|(f,g)\|
_{H^{k_1+3}(\sP)}+1\Big)\|(F,G)\|_{H^{k_2+1}(\sP)}
\|(F,G)\|_{H^{k_3+1}(\sP)}
\\
&\lesssim 
\delta^{-2r}
\Big(\|(f,g)\|
_{H^{r+4}(\sP)}+1\Big)^2
\|(\mu,\nu)\|_{H^{1}(\sP)}^{2}
+\delta \|(F,G)\|_{H^{r}(\sP)}^{2}~.
\end{aligned}
\end{equation}
We have also used that, by assumption, $(f,g)$
is in some small enough neighborhood of $(\bar f,\bar g)$ in $H^5(\sP)$.

The second sum can similarly be estimated as follows: 
\begin{equation}
\begin{aligned}
\label{eq: estimate sum 2}
&\sum  \|H''(f,g)\|_{C^{r-k_4}(\overline{\sP})}
\|(F,G)\|_{H^{k_4}(\sP)}
\|(F,G)\|_{H^{r}(\sP)}
\\
&\lesssim \sum \|H''(f,g)\|_{C(\overline{\sP})}^{\frac{k_4}{r}} \|H''(f,g)\|_{C^r(\overline{\sP})}^{\frac{r-k_4}{r}} 
\|(F,G)\|_{L^2(\sP)}^{\frac{r-k_4}{r}} \|(F,G)\|_{H^r(\sP)}^{\frac{r+k_4}{r}}
\\
&\stackrel{\mathclap{\eqref{eq: FG less munu}}}{\lesssim}
\delta^{-2r}
\|H''(f,g)\|
_{C^r(\overline{\sP})}^2
\|(\mu,\nu)\|_{L^2(\sP)}^{2}
+\delta \|(F,G)\|_{H^{r}(\sP)}^{2}~.
\end{aligned}
\end{equation}

Let us now choose
\begin{equation}
\label{eq: r dependence 2}
\|(\partial_2\nabla f,\partial_3 \nabla f,\partial_2\nabla g,\partial_3\nabla g)
\|_{C^1(\overline \sP)}
<r^{-2}\delta.
\end{equation}
If $\delta$ is small enough (this is allowed by assumption \eqref{eq: with partial derivatives 2 and 3, recurrence}),
then, by \eqref{eq: formula with sum}--\eqref{eq: estimate sum 2}
(note that the coefficient $32$ is replaced by $64$),
\begin{align*}
&
\sum_{j\in\{2,3\}}\int_{\sP}\Big\{ 
\frac 1 {16}(v\innpr \nabla \partial^r_jF)^2
+\frac 1{16}(v\innpr \nabla \partial^r_jG)^2+
\frac {\pi^2\min_{\overline \sP}  v_1^2} {64L^2}
((\partial_j^rF)^2+(\partial_j^rG)^2)\Big\}\,dx\,dy\,dz\\
&\qquad +\|(F,G)\|^2_{L^2(\sP)}\\
&\lesssim
\|(F,G)\|^2_{L^2(\sP)}+
\delta^{-1}\|(\mu,\nu)\|^2_{H^r(\sP)}
+\delta\|(F,G,\partial_1 F,\partial_1G)\|_{H^{r-1}(\sP)}^{2}
\\
&\qquad +\delta^{-2r}
\Big(\|(f,g)\|
_{H^{r+4}(\sP)}
+\|H''(f,g)\|_{C^r(\overline{\sP})}+1\Big)^2
\|(\mu,\nu)\|_{H^{1}(\sP)}^{2}
\end{align*}
because, for $\widehat r =r$, 
\begin{equation}
\label{eq: equivalent norms}
\sum_{|\alpha_2|+|\alpha_3|\leq \widehat r} \|(\partial^\alpha F,\partial^\alpha
G)\|_{L^2(\sP)}^2 
\lesssim
\|(F,G)\|_{L^2(\sP)}^2+\sum_{j\in\{2,3\}}
\|(\partial_j^{\widehat r}F,\partial_j^{\widehat r}G))\|_{L^2(\sP)}^2,
\end{equation}
where the sum is over all multi-indices $\alpha=(\alpha_2,\alpha_3)\in\NN_0^2$
such that $|\alpha_2|+|\alpha_3|\leq \widehat r$ and $\partial^\alpha$ 
is the corresponding partial derivative with respect to the variables $(y,z)$.
Thanks to the induction hypothesis
\begin{multline}
\label{eq: consequence of induction}
\sum_{j\in\{2,3\}}\int_{\sP}\Big\{ 
\frac 1 {16}(v\innpr \nabla \partial^r_jF)^2
+\frac 1{16}(v\innpr \nabla \partial^r_jG)^2
+
\frac {\pi^2\min_{\overline \sP}  v_1^2} {64L^2}
((\partial_j^rF)^2+(\partial_j^rG)^2)\Big\}\,dx\,dy\,dz
+\|(F,G)\|^2_{L^2(\sP)}
\\\lesssim 
\delta^{\,-1}\|(\mu,\nu)\|^2_{H^r(\sP)}
+\delta\|(\partial_1 F,\partial_1G)\|_{H^{r-1}(\sP)}^{2}
\\+\delta^{\,-2r}
\Big(\|(f,g)\|
_{H^{r+4}(\sP)}
+\|H''(f,g)\|_{C^r(\overline{\sP})}+1\Big)^2
\|(\mu,\nu)\|_{H^{1}(\sP)}^{2}.
\end{multline}

\noindent{\bf Second step.}
Let us now deal with the terms containing only one
partial derivative with respect to $x$ and $r-1$  partial derivatives
with respect to $y$ or $z$.
By induction, we know that
\begin{align*}
&\sum_{j\in\{2,3\}}\int_{\sP}\Big\{ 
\frac 1 {16}(v\innpr \nabla \partial^{r-1}_jF)^2
+\frac 1{16}(v\innpr \nabla \partial^{r-1}_jG)^2
\Big\}\,dx\,dy\,dz+\|(F,G)\|_{H^{r-1}(\sP)}^2\\
&
 \leq 
C_{r-1} \|(\mu,\nu)\|_{H^{r-1}(\sP)}^2
+C_{r-1}\|(\mu,\nu)\|_{H^1(\sP)}^2
 \Big(\|(f,g)\|_{H^{r+3}(\sP)}
+\|H''(f,g)\|_{C^{r-1}(\overline{\sP})}+1
\Big)^2
\end{align*}
and thus
\begin{align*}
&\sum_{j\in\{2,3\}} \|(\partial_1\partial^{r-1}_j F,\partial_1\partial^{r-1}_j G)\|_{L^2(\sP)}^2
\\
&\lesssim  \sum_{j\in\{2,3\}}\|(\partial_2\partial^{r-1}_j F, \partial_2\partial^{r-1}_j G, \partial_3\partial^{r-1}_j F, \partial_3\partial^{r-1}_j G)\|_{L^2(\sP)}^2\\
&
\quad+ \|(\mu,\nu)\|_{H^{r-1}(\sP)}^2+\|(\mu,\nu)\|_{H^1(\sP)}^2
 \Big(\|(f,g)\|_{H^{r+3}(\sP)}+\|H''(f,g)\|_{C^{r-1}(\overline{\sP})}+1
\Big)^2
\end{align*}
because the first component of $v$ never vanishes.
Together with the first step
and thanks to \eqref{eq: equivalent norms} with $\widehat r=r$, this gives
\begin{align*}
&\|(F,G)\|_{L^2(\sP)}^2
+\sum_{j\in\{2,3\}} \|(\partial_1\partial^{r-1}_j F, \partial_1\partial^{r-1}_j G)\|_{L^2(\sP)}^2
\\&\lesssim
\delta^{\, -1}\|(\mu,\nu)\|^2_{H^r(\sP)}
+\delta\|(\partial_1 F,\partial_1G)\|_{H^{r-1}(\sP)}^{2}
\\&\quad
+\delta^{\,-2r}\Big(\|(f,g)\|
_{H^{r+4}(\sP)}
+\|H''(f,g)\|_{C^{r}(\overline{\sP})}+1\Big)^2
\|(\mu,\nu)\|_{H^{1}(\sP)}^{2}\,.
\end{align*}
Applying \eqref{eq: equivalent norms} to $\widehat r=r-1$ and
to $(\partial_1 F,\partial_1 G)$, we obtain for small enough 
$\delta$
\begin{equation}
\label{eq: second step inequality}
\begin{aligned}
&\|(F,G)\|_{L^2(\sP)}^2
+\|(\partial_1F,\partial_1G)\|_{L^2(\sP)}^2
+\sum_{j\in\{2,3\}}\|(\partial^{r-1}_j\partial_1 F, \partial^{r-1}_j\partial_1 G)\|_{L^2(\sP)}^2
\\&\lesssim 
\delta^{\, -1}\|(\mu,\nu)\|^2_{H^r(\sP)}
+\delta\|(\partial_1^2 F,\partial_1^2G)\|_{H^{r-2}(\sP)}^{2}
\\&\quad
+\delta^{\, -2r}
\Big(\|(f,g)\|
_{H^{r+4}(\sP)}
+\|H''(f,g)\|_{C^{r}(\overline{\sP})}+1\Big)^2
\|(\mu,\nu)\|_{H^{1}(\sP)}^{2}\,.
\end{aligned}
\end{equation}

\noindent {\bf Third step.}
We now deal with partial derivatives in which
$F$ and $G$ are differentiated at least twice with respect to $x$. 
We estimate these using induction on the number of partial derivatives with respect to $x$ for a fixed $r$.
In the special case $r=2$ 
there is only one second order partial derivative to estimate, and we simply note directly using Proposition \ref{prop: solve second partial in x} that
\begin{align*}
\|(\partial_1^2 F, \partial_1^2 G)\|_{L^{2}(\sP)}
&
\lesssim 
\|(\mu,\nu)\|_{L^2(\sP)}
+\|(\partial_2\nabla F,\partial_2\nabla G, \partial_3\nabla F,\partial_3\nabla G)\|_{L^{2}(\sP)}
+\|(F,G)\|_{H^{1}(\sP)}
\\&\stackrel{\mathclap{\eqref{eq: FG1 less munu1}}}\lesssim 
(\|H''(f,g)\|_{C^{1}(\overline \sP)}+1)
\|(\mu,\nu)\|_{H^1(\sP)}
+\|(\partial_2\nabla F,\partial_2\nabla G,\partial_3\nabla F,\partial_3\nabla G)\|_{L^{2}(\sP)}.
\end{align*}

Next, let $r>2$ and $B_s$ be a differential operator of order $r-2$ in $(x,y,z)$
that consists of an iteration of $r-2$ partial derivatives,
exactly $s$ of which are with respect to $x$ ($0\leq s\leq r-2$).
Differentiating $r-2$ times the expressions for $\partial_1^2 F$ and
$\partial_1^2 G$  in Proposition \ref{prop: solve second partial in x}, we get
\begin{align*}
\|(B_s\partial_1^2 F, B_s\partial_1^2 G)\|_{L^{2}(\sP)}
&\lesssim
\sum_{k=0}^{r-2} (\|(f,g)\|_{H^{r+1-k}(\sP)}+1) \|(\mu, \nu)\|_{H^k(\sP)}
\\&\quad 
+\sum_{k=0}^{r-2} (\|(f,g)\|_{H^{r+2-k}(\sP)}+1) \|(F,G)\|_{H^{k+1}(\sP)}\\
 &\quad 
+\sum_{k=0}^{r-2} \|H''\|_{C^{r-2-k}(\overline{\sP})} \|(F,G)\|_{H^k(\sP)}\\
&\quad +
\Big(\|(f,g)\|_{H^3(\sP)} +1\Big)
\Big(\|D_{s}(\partial_2F,\partial_2G)\|_{L^2(\sP)}+\|E_{s}(\partial_3F,\partial_3G)\|_{L^{2}(\sP)}\Big)
\end{align*}
where $D_{s}$ and $E_{s}$ are matricial differential operators
of order $r-1$ in $(x,y,z)$, but at most of order $s+1$ when seen
as differential operators in $x$ (their coefficients being constants).
The terms involving $E_s$ and $D_s$ come from applying $B_s$ to the terms
in Proposition \ref{prop: solve second partial in x}
involving $\partial^2_{\alpha \beta}F$ or $\partial^2_{\alpha \beta}G$ 
with $(\alpha,\beta)\neq (1,1)$.
The last inequality allows one to estimate  differential expressions
of order $s+2$ with respect to $x$ by differential expressions
of orders at most $s+1$ with respect to $x$.

We get again by interpolation and Young's inequality
\begin{align*}
&
\|(B_s\partial_1^2 F, B_s\partial_1^2 G)\|_{L^{2}(\sP)}
\\&\lesssim 
(\|(f,g)\|_{H^{r+1}(\sP)}+1) \|(\mu, \nu)\|_{L^2(\sP)} + (\|(f,g)\|_{H^{3}(\sP)}+1) \|(\mu, \nu)\|_{H^{r-2}(\sP)}
\\&\qquad+
(\|(f,g)\|_{H^{r+3}(\sP)}+1) \|(F,G)\|_{L^2(\sP)} +(\|(f,g)\|_{H^{4}(\sP)}+1) \|(F,G)\|_{H^{r-1}(\sP)} 
\\&\qquad+
\|H''(f,g)\|_{C^{r-2}(\overline{\sP})} \|(F,G)\|_{L^2(\sP)} +\|H''(f,g)\|_{C(\overline{\sP})} \|(F,G)\|_{H^{r-2}(\sP)}
\\
&\qquad +\Big(\|(f,g)\|_{H^3(\sP)} +1\Big)
\Big(\|D_{s}(\partial_2F,\partial_2G)\|_{L^2(\sP)}+\|E_{s}(\partial_3F,\partial_3G)\|_{L^{2}(\sP)}\Big)\\
&\lesssim 
\Big(\|(f,g)\|_{H^{r+3}(\sP)}+\|H''(f,g)\|_{C^{r-2}(\overline{\sP})}+1
\Big)
\|(\mu,\nu)\|_{L^{2}(\sP)}
\\&\quad
+\|(\mu,\nu)\|_{H^{r-2}(\sP)}
+\|(F,G)\|_{H^{r-1}(\sP)}+\|D_{s}(\partial_2F,\partial_2G)\|_{L^2(\sP)}
+\|E_{s}(\partial_3F,\partial_3G)\|_{L^{2}(\sP)}
\\&\lesssim 
 \Big(\|(f,g)\|_{H^{r+3}(\sP)}
+\|H''(f,g)\|_{C^{r-1}(\overline{\sP})}+1
\Big)
\|(\mu,\nu)\|_{H^{1}(\sP)}
\\&\quad
+\|(\mu,\nu)\|_{H^{r-1}(\sP)}
+\|D_{s}(\partial_2F,\partial_2G)\|_{L^2(\sP)}+\|E_{s}(\partial_3F,\partial_3G)\|_{L^{2}(\sP)},
\end{align*}
where we've used the induction hypothesis \eqref{eq: tame inverse} with $r$ replaced by $r-1$ in the last step.
By induction on $s$, we get the estimate
\begin{align*}
\|(B_s\partial_1^2 F, B_s\partial_1^2 G)\|_{L^{2}(\sP)}
&\lesssim 
 \Big(\|(f,g)\|_{H^{r+4}(\sP)}
+\|H''(f,g)\|_{C^{r}(\overline{\sP})}+1
\Big)
\|(\mu,\nu)\|_{H^{1}(\sP)}
\\
&\quad
+\|(\mu,\nu)\|_{H^{r}(\sP)}
+\sum_{j\in\{2,3\}}\|(\partial^{r-1}_j\partial_1 F, \partial^{r-1}_j\partial_1 G)\|_{L^2(\sP)}\\
&\quad \quad+\delta \|(\partial_1^2 F, \partial_1^2 G)\|_{H^{r-2}(\sP)},
\end{align*}
thanks to \eqref{eq: equivalent norms} applied to $(F,G)$ and $(\partial_1F,\partial_1G)$, and to \eqref{eq: consequence of induction}.
Hence, choosing $\delta$ sufficiently small
\begin{equation}
\label{eq: third step inequality}
\begin{aligned}
\|(\partial_1^2 F, \partial_1^2 G)\|_{H^{r-2}(\sP)}
&\lesssim  
\Big(\|(f,g)\|_{H^{r+4}(\sP)}
+\|H''(f,g)\|_{C^{r}(\overline{\sP})}+1
\Big)\|(\mu,\nu)\|_{H^{1}(\sP)}
\\&
\quad
+\|(\mu,\nu)\|_{H^{r}(\sP)}
+\sum_{j\in\{2,3\}}\|(\partial^{r-1}_j\partial_1 F, \partial^{r-1}_j\partial_1 G)\|_{L^2(\sP)}.
\end{aligned}
\end{equation}
Combining \eqref{eq: third step inequality} with \eqref{eq: second step inequality} and again choosing $\delta$ sufficiently small allows us to estimate all partial derivatives of order $r$ with precisely one derivative with respect to $x$.
Substitution of the resulting estimate into \eqref{eq: third step inequality} gives us control of all derivatives with at least two derivatives with respect to $x$.

\noindent{\bf Conclusion.}
The estimate of the statement follows from the three steps.
\end{proof}

Let us deal with the case $\epsilon=0$ with the help of
the technique of elliptic regularization  introduced and well explained
in \cite{KoNi}, see e.g p.\ 449, the beginning of the proof of Theorem 2
and the proof of Theorem 2' in that work.
Firstly, when $\epsilon>0$, one deduces from this a priori estimate
the existence of an admissible solution $(F,G)\in H^{r}(\sP)$ given any
$(\mu,\nu)\in H^r(\sP)$,
by approximating $(f,g)$, $H''(f,g)$ itself and $(\mu,\nu)$
by smooth functions.
The existence of $(F,G)$ is ensured because the problem is elliptic in this case.
Secondly, as the above estimate holds uniformly in 
$\epsilon\in(0,1]$, 
the existence persists when taking the limit $\epsilon\rightarrow 0$.
Thus we get the following theorem.

\begin{thm}
\label{thm: general estimate}
Let $\epsilon=0$,
$r\geq 1$ be an integer,
$(f,g)\in H^{r+4}_{loc}(D)$ (admissible)
be in some small enough neighborhood of $(\bar f,\bar g)$ in  
$H^5(\sP)$,
$H\in C^{2}(\RR^2)$ be admissible, 
$H''(f,g)\in C^{r}(\overline{D})$ and
$H''(f,g)$ be small enough in $C(\overline{\sP})$.
There exists a constant $C_r>0$ such that if
$$
\|(\partial_2\nabla f,\partial_3 \nabla f,\partial_2\nabla g,\partial_3 \nabla g)
\|_{C^1(\overline \sP)}
<C_r^{-1},
$$
then for any periodic $(\mu,\nu)\in H^r_{loc}(D)$ there exists  
an admissible $(F,G)\in H^{r}_{loc}(D)$ satisfying
\eqref{weak linear problem} (with $\epsilon=0$) and
\begin{equation*}
\|(F,G)\|_{H^r(\sP)}^2
 \leq 
C_r \|(\mu,\nu)\|_{H^r(\sP)}^2
+C_r\|(\mu,\nu)\|_{H^1(\sP)}^2
 \Big(\|(f,g)\|_{H^{r+4}(\sP)}
+\|H''(f,g)\|_{C^{r}(\overline{\sP})}
+1\Big)^2~.
\end{equation*}
\end{thm}
This result remains true without the
simplifying hypothesis \eqref{eq: for simplicity}.

\section{A solution by the Nash-Moser method}
\label{sec: Nash-Moser}

In this section we shall take $\bar f$ and $\bar g$ to be some fixed linear functions and let $R$ be the corresponding Jacobian matrix with respect to $(y,z)$ as in the Introduction.

Let us define three decreasing sequences of Banach spaces.

\noindent
{\bf Definition of the Banach spaces $\sU_k$.}
For each integer $k\geq 2$, 
let $\sU_{k}$ be the real linear space
of all $(F,G)$ in $H^{k}_{loc}(D)$
satisfying (Ad'2) and (Ad'3).  
We define the norm $\|\cdot\|_k$ on $\sU_k$  as
$$\|(F,G)\|^2_{k}=\|F\|_{H^{k}(\sP)}^2
+\|G\|_{H^{k}(\sP)}^2~.$$

\noindent
{\bf Definition of the Banach spaces $\sV_k$.}
For each integer $k\geq 0$, 
let $\sV_{k}$ be the real linear space
of all $(\mu,\nu)$ in $H^{k}_{loc}(D)$ that satisfy
the periodicity condition (Ad'2)  almost everywhere. 
We define the norm $\|\cdot\|_{k}$ on $\sV_k$ by
$$\|(\mu,\nu)\|^2_{k}=\|\mu\|_{H^{k}(\sP)}^2
+\|\nu\|_{H^{k}(\sP)}^2~.$$

\noindent
{\bf Definition of the Banach spaces $\sW_k$.}
For each integer $k\geq 4$,
let $\sW_{k}$ be the real linear space
of $(f_0,g_0,H_0,c)$ such that
\begin{itemize}
\item[(i)]
$f_0,g_0\in H^{k}_{loc}(D)$ satisfy  the periodicity condition (Ad'2),
\item[(ii)]
$H_0\in C^{k-2}(\RR^2)$ is periodic  with respect to the lattice generated by $RP_1e_1$ and $RP_2e_2$, and $c\in \RR^2$.
\end{itemize}
Note that
(ii)  ensures that 
$H_0(\bar f+f_0+f_1, \bar g+g_0+g_1)$
satisfies (Ad'2) for all $(f_1,g_1)\in \sU_k$.

We define the norm $\|\cdot\|_{k}$ on $\sW_k$ by
\begin{equation*}
\|(f_0,g_0,H_0,c)\|^2_{k}
=\|f_0\|_{H^{k}(\sP)}^2
+\|g_0\|_{H^{k}(\sP)}^2
+\|H_0\|^2_{C^{k-2}(\overline{\sQ})}+|c|^2\,.
\end{equation*}

\vspace{2mm}

Given $(f_0,g_0,H_0,c)\in \sW_4$, with $H_0\in C^{3}(\RR^2)$,
we define the map
$\sF\colon \sU_4\rightarrow \sV_2$ by
$$\left(\begin{array}{c}f_1\\g_1\end{array}\right)\rightarrow
\sF\left(\begin{array}{c}f_1\\g_1\end{array}\right)
=\left(\begin{array}{c}
-\Div(\nabla  g \times (\nabla  f\times \nabla  g))+
 \partial_fH(f,g)
\\
-\Div((\nabla  f\times \nabla  g)\times \nabla  f))+\partial_g H(f,g)
\end{array}\right)
$$
with
$f=\bar f+f_0+f_1$, $g=\bar g +g_0+g_1$ and $H(f,g)=c_1f+c_2g+H_0(f,g)$.

The following theorem results directly from Theorem \ref{thm: general estimate} and 
\eqref{eq: composition} (with $\xi=H$).
\begin{thm}
\label{thm: summary}
Let $k\geq 1$ be an integer
and suppose that $(f_0,g_0,H_0,c)\in \sW_{k+4}$, 
$(f_1,g_1)\in \sU_{k+4}$,
$\|H_0''\|_{C(\overline{\mathcal Q})}$ 
is small enough, and 
 $(f,g)$ is in some small enough neighborhood of
$(\bar f,\bar g)$ in  $H^5(\sP)$, with
$$
f=\bar f +f_0+f_1, ~~g=\bar g +g_0+g_1
~\text{ and }~H(f,g)=c_1f+c_2g+H_0(f,g).
$$
There exists a constant $M_k>0$ such that if
$$
\|(\partial_2\nabla f,\partial_3 \nabla f,\partial_2\nabla g,\partial_3\nabla g)
\|_{C^1(\overline \sP)}
<M_k^{-1}
$$
we get the following.
Given any
$(\mu,\nu)\in \sV_{k}$, there exists a unique
$(F,G)\in \sU_{k}$ satisfying 
\eqref{weak linear problem} with $\epsilon=0$.
It also satisfies
\begin{equation*}
\|(F,G)\|_ {k}\leq M_k \|(\mu,\nu)\|_{k}+M_k\|(\mu,\nu)\|_{1}\Big(\|(f_1,g_1)\|_{H^{k+4}(\sP)}+1\Big)
\end{equation*}
and
$$\|(F,G)\|_0 \leq M_0 \|(\mu,\nu)\|_{0}$$
for some constant $M_0>0$ independent of $k$.
\end{thm}

\begin{rem}
The constants $M_k$ in Theorem \ref{thm: summary} can also depend on $(f_0,g_0,H_0,c)$ and $(\bar f,\bar g)$.
\end{rem}

 Let us state Theorem 6.3.1 in \cite{Han-Hong}.
There $\Omega$ is a smooth domain in $\RR^n$ or a rectangle
with the sides parallel to the coordinate axes and with
periodic boundary conditions with respect to $n-1$ coordinates.
The corresponding Sobolev spaces are simply denoted by $H^{k}$. 
\begin{thm}
\label{thm: Nash-Moser}
Suppose $\sF(w)$ is a nonlinear differential operator of
order $m$ in a domain $\Omega\subset \RR^n$, given by
$$\sF(w)=\Gamma(x,w,\partial w,\ldots,\partial^m w),$$
where $\Gamma$ is smooth (see however the remark below).

Suppose that $d_0,d_1,d_2,d_3,s_0$ and $\widetilde s$ 
are non-negative integers with
$$d_0\geq m+[n/2]+1$$
and
$$\widetilde s\geq\max \{3m+2d_*+[n/2]+2,m+d_*+d_0+1,m+d_2+d_3+1\},$$
where $d_*=\max\{d_1,d_3-s_0-1\}$. 
Assume that, for any $h\in H^{\widetilde s +d_1}
= H^{\widetilde s +d_1}(\Omega)$
 and $w\in H^{\widetilde s +d_2}$ with
$$\|w\|_{H^{d_0}}\leq r_0\coloneqq 1,$$
the linear equation 
\begin{equation}
\label{eq: abstract linear}
\sF'(w)\rho=h
\end{equation}
admits a solution $\rho\in H^{\widetilde s}$ satisfying for any $s=0,1,\ldots,\widetilde s$
$$\|\rho\|_{H^s}\leq c_s\left(\|h\|_{H^{s+d_1}}+(s-s_0)^+(\|w\|_{H^{s+d_2}}+1)
\|h\|_{H^{d_3}}\right),$$
where $c_s$ is a positive constant independent of $h,w$ and $\rho$.
Then there exists a positive constant $\mu_*$, depending only on
$\Omega,c_s,m,d_0,d_1,d_2,d_3,s_0$ and $\widetilde s$, such that if
\begin{equation}
\label{eq: mu* condition}
\|\sF(0)\|_{H^{\widetilde s-m}}\leq \mu_*^2,
\end{equation}
the equation $\sF(w)=0$ admits an $H^{\widetilde s -m-d_*-1}$
solution $w$ in $\Omega$.
\end{thm}

\begin{rems}$ $
\begin{itemize}
\item
By inspecting the proof in \cite{Han-Hong}, 
we see that it holds as well for systems of $N\geq 1$ differential equations.
Moreover the  constant $r_0=1$ can be replaced by any fixed value $r_0>0$
by multiplying appropriately functions by constant factors.
\item
Also the solution $w$ is the limit in $H^{\widetilde s-m-d_*-1}$ of  sums
of solutions in $H^{\widetilde s}$ to linear equations of type \eqref{eq: abstract linear}. 
See in \cite{Han-Hong} equations (6.3.14) and (6.3.15), and
the proof of Theorem 6.3.1 on p.\ 103.
\item
We can relax the condition that $\Gamma$ is smooth.
Let $\widehat c>0$ be such that, for all $ w\in H^{d_0}$ with
$\|w\|_{H^{d_0}}\leq r_0$, we have
$$\|w\|_{C^m(\overline \Omega)} \leq \widehat c,$$
and define $\Sigma\subset\RR^{N+Nn+Nn^2+\ldots Nn^m}$ as the ball
of radius $\widehat c$ centered at the origin. 
In the proof, the map $\Gamma$  appears in the various
estimates via $\|\sF(0)\|_{H^{\widetilde s-m}}$ 
and  via   ``constants'' depending on
$$\|\partial_{\alpha}\partial_\beta\Gamma\|
_{C^{\widetilde s-m}(\overline \Omega\times \overline\Sigma)},
$$
where $\partial_\alpha$ and $\partial_\beta$
are all possible partial derivatives with respect to
$w,\ldots,\partial^mw$.
See \eqref{eq: composition} and,
in \cite{Han-Hong}, the proof of $(P_3)_{\ell+1}$ on p.\ 101. It therefore suffices to assume that $\Gamma$ is of class $C^{\tilde s-m+2}$.
\item 
From \cite{Han-Hong} it follows that there exists a constant $C>0$ such that $\|w\|_{H^{\tilde s-m-d*-1}}\le C\mu_*^2$. 
More precisely, see in \cite{Han-Hong} the last estimate in the proof of
$(P_1)_{l+1}$ on p.\ 100,  (6.3.31) and the proof of
Theorem 6.3.1 on p.\ 103.
\end{itemize}
\end{rems}

To apply this theorem,
we need to check \eqref{eq: mu* condition}.
For this reason, we shall  stay near a solution (namely $(f_1,g_1)=0$) 
to an unperturbed problem (namely $(f_0,g_0)=0$ and $H=0$), so that
\eqref{eq: mu* condition} is satisfied,
and rely on the fact that all relevant ``constants'' (in particular $\mu^*$)  
for the perturbed
problem can be chosen equal to those of the unperturbed problem.
\begin{thm} 
\label{thm: main result}
Let $j\ge 0$ be an integer, $R>0$ arbitrary and $\delta>0$ sufficiently small
and assume that $(f_0, g_0, H_0, c)\in \sW_{13+j}$ with
 $\|(f_0,g_0,H_0,c)\|_{13+j}$ $<R$ and $\|(f_0, g_0, H_0, 0)\|_{5}<\delta$.
It is possible to choose $\epsilon>0$ (independent of
$(f_0,g_0,H_0,c)$, but depending on $(\bar f, \bar g)$, $j$, $R$ and $\delta$) such  that
if $\|\sF(0,0)\|_{7+j}<\epsilon$ then there exists
$(f^*,g^*)\in \sU_{6+j}$ satisfying
$\sF(f^*,g^*)=0$.
\end{thm}
\begin{proof}
We choose $r_0>0$ small enough so that
Theorem \ref{thm: summary} with $k=9+j$ can be applied
for all $ (f_1,g_1)\in \sU_5$ in the closed ball of radius $r_0$ 
centered at the origin.
Let $\widehat c>0$ be such that
$$\|(f_1,g_1)\|_{C^2(\overline \sP)} \leq \widehat c $$
for all $ (f_1,g_1)\in \sU_5$ in this ball, 
and define $\Sigma\subset\RR^{2+6+18}$ 
as the ball of radius $\widehat c$ centered at the origin.
 
We apply Theorem \ref{thm: Nash-Moser}
with $m=2$,
$\Omega=\sP\subset \RR^n$, $n=3$,
$d_0=5$, $d_1=0$, $d_2=4$, $d_3=1$, $s_0=1$, $d_*=0$ and  $\widetilde s=9+j$.
We get $\widetilde s+d_1=9+j$,
$\widetilde s+d_2=13+j$,
$\widetilde s-m=7+j$,
$\widetilde s-m-d_*-1=6+j$
and a solution $(f^*,g^*)\in H^{6+j}(\sP)$.
Let the map $\Gamma\colon \sP\times\RR^{1+1+3+3+9+9}
\rightarrow\RR^2$
be such that
$$\sF(f_1,g_1)=\Gamma(x,y,z,f_1,g_1,f_1',g_1',f_1'',g_1'').$$
It appears in the various
estimates also via   ``constants'' depending on
$\|\partial_{\alpha}\partial_\beta\Gamma\|
_{C^{\widetilde s-m}(\overline \sP\times \overline\Sigma)}$,
where $\partial_\alpha$ and $\partial_\beta$
are all possible partial derivatives with respect to
$f_1,g_1,f_1',g_1',f_1''$ or $g_1''$.
Observe that 
$(f_0,g_0,H_0,c)\in \sW_{13+j}$ implies
$(f_0,g_0,H_0,c)\in C^{\widetilde s+2}(\overline \sP)
\times C^{\widetilde s+2}(\overline \sP)
\times C^{\widetilde s+2}(\overline \sQ)\times \RR^2$ and
$\partial_{\alpha}\partial_\beta\Gamma \in
C^{\widetilde s-m}(\overline \sP\times \overline \Sigma)$.
As $(f^*,g^*)$ is the limit in $H^{6+j}(\sP)$ 
of sums of solutions in $\sU_{9+j}$ to equations of type
\eqref{weak linear problem} (with $\epsilon=0$), it satisfies (Ad'3) and thus belongs to $\sU_{6+j}$.
\end{proof}

As a corollary, we get the following simplified statement.

\begin{thm} 
\label{thm: corollary}
Assume that $H_0\in C^{11+j}$ and $f_0, g_0 \in H^{13+j}$.
It is possible to choose $\bar\epsilon>0$ such that
if $\|(f_0,g_0,H_0,c)\|_{13+j} <\bar \epsilon$,
 then there exists
$(f^*,g^*)\in \sU_{6+j}$ satisfying
$\sF(f^*,g^*)=0$.
\end{thm}
Theorem \ref{thm: main in intro} is a reformulation of this last result  and Theorem \ref{thm: uniqueness}.

\appendix

\section*{Appendix: Representation of divergence free vector fields}

The fact that the vector field $\nabla f\times \nabla g$ is divergence free if $f$ and $g$ are $C^2$ is easily checked using the formula $\Div(u \times v)=v\innpr \rot u-u\innpr \rot v$. A local converse near points where $v$ is non-zero has been known for a long time; see e.g.\ \cite{Barbarosie} and \cite{Ca:1967} (Chapter 3, exercise 14).
A local converse that can be seen as a global converse
under additional conditions can be found in Appendix I in
\cite{GrRu}.
In the present appendix, we give for the reader's convenience
a self-contained proof that a divergence free vector field $v\in C^2(\overline D)$ can be represented globally in this form 
if $v$ is periodic in $y$ and $z$ and $v_1 \ne 0$ in $\overline D$,
and that $f$ and $g$ can be chosen to be of the form ``linear plus periodic". Our argument is essentially a simple version of an elementary proof of global equivalence of volume forms on compact connected manifolds due to Moser \cite{Mo:65}.

For a given point $(x,y,z)\in \overline{D}$ we solve the system of ODEs $\phi'=v(\phi)$, with $\phi(0)=(x,y,z)$, and let 
$T=T(x,y,z)$ be the unique time such that $\phi_1(-T; x,y,z)=0$ (here we use that $\inf_{\overline D} |v_1|>0$ and $\sup_{\overline D} |v|<\infty$). We define the $C^2$ functions $Y, Z\colon \overline{D} \to \RR^2$ by
$$
Y\colon (x,y,z)\mapsto \phi_2(-T; x, y, z) \quad \text{and} \quad 
Z\colon (x,y,z)\mapsto \phi_3(-T; x,y, z).
$$
The functions $Y$ and $Z$ are invariants of the vector field $v$  and therefore $\nabla Y\times \nabla Z=\lambda v$ for some function $\lambda$. Using the fact that $v$ is divergence free, it is easily established that $\lambda$ is another invariant and therefore
\[
\nabla Y \times \nabla Z=\frac{1}{v_1(0,Y,Z)} v 
\]
in view of the relations $Y(0,y,z)=y$ and $Z(0,y,z)=z$. If $F, G\colon \RR^2 \to \RR^2$ and 
$$
f(x,y,z)=F(Y(x,y,z),Z(x,y,z)), \qquad g(x,y,z)=G(Y(x,y,z), Z(x,y,z)),
$$
then
$$
\nabla f\times \nabla g=(\partial_1 F \partial_2 G-\partial_2 F  \partial_1 G) \nabla Y \times \nabla Z.
$$
Thus in order to have $\nabla f\times \nabla g=v$ we must find $F$ and $G$ with
$$
\partial_1 F(Y,Z) \partial_2 G(Y,Z)-\partial_2 F(Y,Z)  \partial_1 G(Y,Z)=v_1(0,Y,Z).
$$
If it weren't for the periodicity conditions,  this would be trivial. We describe next how to make a choice which respects these conditions (the choice is not unique).

Note that $v_1(0,Y,Z)$ is $P_1$-periodic in $Y$ and $P_2$-periodic in $Z$. Let 
$$
\alpha=\frac{1}{P_1P_2}\int_0^{P_1} \int_0^{P_2} v_1(0, Y,Z)\, dY\, dZ
$$
and write $v_1(0,Y,Z)=a(Y)b(Y,Z)$,
where
$$
a(Y)=\frac{1}{P_2} \int_0^{P_2} v_1(0,Y,Z)\, dZ \qquad \text{and} \qquad
b(Y,Z)=\frac{v_1(0,Y,Z)}{a(Y)},
$$
so that
$$
\frac{1}{P_1}\int_0^{P_1} a(Y) \, dY= \alpha
\qquad \text{and} \qquad 
\frac{1}{P_2}\int_0^{P_2} b(Y,Z) \, dZ=1.
$$
We choose
$$
F(Y)=\int_0^Y a(s)\, ds \qquad \text{and} \qquad G(Y,Z)=\int_0^Z b(Y, s) \, ds.
$$
Note that $F$ and $G$ (and hence $f$ and $g$) are $C^2$ and that the map 
\[
\Psi\colon (Y,Z)\mapsto (F(Y),G(Y,Z))
\]
from $\RR^2$ to itself is bijective.
It is easily verified that
$$
\partial_1 F(Y) \partial_2 G(Y,Z)=a(Y)b(Y,Z)=v_1(0, Y, Z),
$$
that $F(Y)-\alpha Y$ is $P_1$-periodic and that $G(Y,Z)-Z$ is $(P_1, P_2)$-periodic. Finally, by the periodicity of $v$ and standard ODE theory, it follows that $(Y(x,y,z), Z(x,y,z))-(y,z)$ is $P_1$ periodic in $y$ and $P_2$-periodic in $z$, and therefore so is $(f(x,y,z),g(x,y,z))-(\alpha y, z)$. This concludes the proof.

As mentioned above, the representation $v=\nabla f\times \nabla g$ is not unique. Indeed, if $\Phi \in C^2(\RR^2,\RR^2)$ satisfies
\[
\det \Phi' = \partial_1\Phi_1 \partial_2 \Phi_2-\partial_2 \Phi_1 \partial_1 \Phi_2=1,
\]
then $(\tilde f, \tilde g) =\Phi(f,g)$ also satisfies $\nabla \tilde f\times \nabla \tilde g=v$. Moreover, $(\tilde f, \tilde g)$ is also linear plus $(P_1,P_2)$-periodic in $(y,z)$ if $\Phi(f,g)=T(f,g)+\Phi_0(f,g)$, where $T\colon \RR^2\to \RR^2$ is linear and $\Phi_0$ is $(\alpha P_1, P_2)$-periodic.

Note that $T$ is bijective, since otherwise one could find a non-zero linear functional $\ell$ annihilating its range. This would cause $\ell \circ \Phi$ to be periodic, and thus $\ell\circ\Phi$ would have a critical point at which
$\det \Phi'$ would vanish.
As $T$ is bijective, $\Phi$ is proper and hence bijective by the global inversion theorem (using again $\det\Phi'=1$).

Conversely, if $v=\nabla \tilde f\times \nabla \tilde g$ for some $C^2$ functions $\tilde f$ and $\tilde g$, then $\tilde f$ and $\tilde g$ are constant along the streamlines of $v$. 
Hence $(\tilde f(x,y,z), \tilde g(x,y,z))=(\tilde f(0,Y,Z), \tilde g(0,Y,Z))$ with $(Y,Z)=(Y(x,y,z),Z(x,y,z))$ as above, and we obtain
$(\tilde f, \tilde g)=\Phi(f,g)$, where $\Phi=(\tilde f, \tilde g)|_{x=0} \circ \Psi^{-1}$ is $C^2$. Moreover, $\Phi$ is linear plus $(\alpha P_1, P_2)$-periodic and $\det \Phi'=1$.

Let us finally note that the Bernoulli function $H=\frac12 |v|^2+P$ can clearly be written as a function of $(f,g)$ since it is constant on streamlines. Denoting this function also by $H(f,g)$, we find that if $(f,g)$ is transformed to $(\tilde f, \tilde g)=\Phi(f,g)$ with $\Phi$ as above, then $H$ is transformed to $H\circ \Phi^{-1}$.

\bigskip

\noindent {\bf Acknowledgments}. The authors would like to thank the Isaac Newton Institute for Mathematical Sciences, Cambridge, for support and hospitality during the programme `Nonlinear water waves' where work on this paper was undertaken. This work was supported by EPSRC grant no EP/K032208/1. E. Wahl\'en has received funding from the European Research Council (ERC) under the European Union's Horizon 2020 research and innovation programme (grant agreement no 678698). 
Finally, the authors are grateful to the referees for helpful comments.

\pagebreak

\end{document}